\documentclass[11pt,a4paper]{article}

\usepackage{isorot}

\usepackage{graphicx,subfigure}

\usepackage[square,numbers]{natbib}

\usepackage{booktabs}
\usepackage{multirow}

\usepackage{pstricks,pst-plot}

\usepackage{amsmath}
\usepackage{amsthm}
\usepackage{amsfonts}
\usepackage{amssymb}
\usepackage{latexsym}
\usepackage{stackrel}
\usepackage[T1]{fontenc}
\usepackage{IEEEtrantools}
\usepackage{enumerate}
\usepackage{fourier}
\usepackage{skak}

\usepackage{lastpage}

\let\mathnumsetfont\mathbb
\newcommand\Nset{\mathnumsetfont N} 
\newcommand\Rset{\mathnumsetfont R} 

\newcommand\M{{\cal M}}

\newcommand\Oset{{\cal O}}
\newcommand\fty{{ }}
\newcommand\thelr{the L'H\^{o}pital's rule}
\newcommand\pyp{Peter and Paul distribution}

\newtheorem{teo}{Theorem}[section]
\newtheorem{teo*}{Theorem}
\newtheorem{prop}{Proposition}[section]
\newtheorem{cor}{Corollary}[section]
\newtheorem{rmq}{Remark}[section]
\newtheorem{defi}{Definition}[section]
\newtheorem{exm}{\emph{Example}}[section]
\newtheorem{lem}{Lemma}[section]

\newtheorem{propt}{Properties}[section]

\newcommand\PBdH{Pickands-Balkema-de~Haan theorem}

\def\limsup{\mathop{\overline{\mathrm{lim}}}}
\def\liminf{\mathop{\underline{\mathrm{lim}}}}
\def\limx{\lim_{x\rightarrow\infty}}
\def\limy{\lim_{y\rightarrow\infty}}
\def\limn{\lim_{n\rightarrow\infty}}

\def\limsupx{\limsup_{x\rightarrow\infty}}
\def\liminfx{\liminf_{x\rightarrow\infty}}

\def\barF{\overline{F}}

\title{An extension of the class of regularly varying functions}

\author{Meitner Cadena\thanks{UPMC Paris 6 \& CREAR, ESSEC Business School;\, E-mail: meitner.cadena@etu.upmc.fr or b00454799@essec.edu
}\;\; and\,  Marie Kratz\thanks{ESSEC Business School, CREAR risk research center; \, E-mail: kratz@essec.edu $\qquad\qquad\qquad\qquad\qquad\qquad$ Marie Kratz is also member of MAP5, UMR 8145, Univ. Paris Descartes, France.  }} 

\date{}

\begin{document}

\maketitle

\begin{abstract}
We define a new class of positive and Lebesgue measurable functions in terms of their asymptotic behavior, which includes the class of regularly varying functions. We also characterize it by transformations, corresponding to generalized moments when these functions are random variables. We study the properties of this new class and discuss their applications to Extreme Value Theory. 

\vspace{2ex}

{\it Keywords: asymptotic behavior, domains of attraction; extreme value theory; Karamata's representation theorem; Karamata's theorem; Karamata's tauberian theorem; measurable functions; von Mises' conditions; Peter and Paul distribution; regularly varying function} 

\vspace{2ex}

{\it AMS classification}:  26A42; 60F99; 60G70
\end{abstract}

\section*{Introduction}

The field of Extreme Value Theory (EVT) started to be developed in the 20's, concurrently with the development of modern probability theory by Kolmogorov, 
with the pioneers Fisher and Tippett (1928) who introduced the fundamental theorem of EVT, the Fisher-Tippett Theorem, giving three types of  limit distribution 
for the extremes (minimum or maximum).
A few years later, in the 30's, Karamata defined the notion of slowly varying and regularly varying (RV) functions, describing a specific asymptotic behavior of these functions, namely:\\
{\it Definition.} A Lebesgue-measurable function $U:\Rset^+\rightarrow\Rset^+$ is RV at infinity if,  for all $t>0$,
\begin{equation}\label{eq:000rv}
\lim_{x\rightarrow\infty}\frac{U(xt)}{U(x)}=t^{\rho}\quad\textrm{for some $\rho\in\Rset$},
\end{equation}
$\rho$ being called the tail index of $U$, and the case $\rho=0$ corresponding to the notion of slowly varying function.
$U$ is RV at $0^+$ if \eqref{eq:000rv} holds taking the limit $x\to0^+$. \\
We had to wait for more than one decade to see links appearing between EVT and RV functions.  
Following the earlier works by Gnedenko (see \cite{Gnedenko1943}), 
then Feller (see \cite{feller21966}), who characterized the domains of attraction of Fr\'echet and Weibull using RV functions at infinity, without using Karamata theory in the case of Gnedenko, 
de Haan (1970) generalized the results using Karamata theory and  completed it, providing a complete solution for the case of Gumbel limits. Since then, much work has been developed on EVT and RV functions,  in particular in the multivariate case with the notion of multivariate regular variation (see e.g. \cite{deHaan}, \cite{deHaanFerreira}, \cite{resnick3}, \cite{Resnick2004}, and references therein).

Nevertheless, the RV class may still be restrictive, particularly in practice. If the limit in \eqref{eq:000rv} does not exist, all standard results given for RV functions and used in EVT, as e.g. Karamata theorems, Von Mises conditions, etc...,  cannot be applied.
Hence the natural question of extending this class and  EVT characterizations, for broader applications in view of (tail) modelling. 

We answer this concern in real analysis and EVT, constructing a (strictly) larger class of functions than the RV  class on  which
we generalize EVT results and provide conditions easy to check in practice. 

The paper is organized in two main parts. The first section defines our new large class of functions described in terms of their asymptotic behaviors, which may violate  \eqref{eq:000rv}.
It provides its algebraic properties, as well as characteristic representation theorems, one being of Karamata type.  
In the second section, we discuss extensions for this class of functions of other important Karamata theorems,  and end with results on domains of attraction.
Proofs of the results are given in the appendix.

This study is the first of a series of two papers, extending the class of regularly varying functions. It addresses the probabilistic analysis of our new class. The second paper will treat the statistical aspect of it.

\section{Study of a new class of functions}

We focus on the new class $\M$ of positive and measurable functions with support $\Rset^+$, characterizing their behavior at $\infty$ with respect to polynomial functions.
A number of properties of this class are studied and characterizations are provided.
Further, variants of this class,  considering asymptotic behaviors of exponential type instead of polynomial one, provide other classes, denoted by $\M_\infty$ and $\M_{-\infty}$, having  similar properties and characterizations as $\M$ does.

Let us introduce a few notations.

When using limits, we will discriminate between existing limits, namely finite or infinite ($\infty$, $-\infty$) ones, and not existing ones.

The notation a.s. (almost surely) in (in)equalities concerning measurable functions is omitted.
Moreover, for any random variable (rv) $X$, we denote its distribution by $F_X(x)=P(X\leq x)$, and its tail of distribution by $\barF_X=1-F_X$.
The subscript $X$ will be omitted when no possible confusion.

RV (RV$_\rho$ respectively) denotes indifferently the class of regularly varying functions (with tail index $\rho$, respectively) or the property of regularly varying function (with tail index $\rho$).

Finally recall the notations $\min(a,b)=a\wedge b$ and $\max(a,b)=a\vee b$ that will be used, $\left\lfloor x\right\rfloor$ for the largest integer not greater than $x$
and $\left\lceil x\right\rceil$ for the lowest integer greater or equal than $x$, and $\log(x)$ represents the natural logarithm of $x$.

\subsection{The class $\mathcal{M}$} 

We introduce a new class $\mathcal{M}$ that we define as follows.

\begin{defi}\label{eq:main:defi:classM}
$\mathcal{M}$ is the class of positive and measurable functions $U$ with support $\Rset^+$, bounded on finite intervals, such that 
\begin{equation}\label{Mkappa}
\exists ~\rho\in\Rset, \, \forall \varepsilon > 0,\;  \lim_{x\rightarrow\infty}\frac{U(x)}{x^{\rho+\epsilon}}=0\text{\quad and\quad}\lim_{x\rightarrow\infty}\frac{U(x)}{x^{\rho-\epsilon}}=\infty\,.
\end{equation}
\end{defi}

On $\M$, we can define specific properties.
\begin{propt}\label{propt:main:001}~

\begin{itemize}
\item[(i)]
For any $U\in\M_\fty$, $\rho$ defined in (\ref{Mkappa}) is unique, and denoted by $\rho_U$.

\item[(ii)]
Let $U,V\in\mathcal{M}$ s.t. 
$\rho_U>\rho_V$.
Then $\displaystyle \lim_{x\rightarrow\infty}\frac{V(x)}{U(x)}=0$.

\item[(iii)]
For any  $U,V\in\M$ and any $a\ge 0$,  $aU+V\in\M$ with $\rho_{aU+V}=\rho_U\vee\rho_V$. 

\item[(iv)]  If $U\in\M$ with $\rho_U$ defined in (\ref{Mkappa}), then $1/U\in\M_\fty$ with $\rho_{1/U}=-\rho_U$.

\item[(v)]
Let $U\in\M$ with $\rho_U$ defined in \eqref{Mkappa}.
If $\rho_U<-1$, then $U$ is integrable on $\Rset^+$, whereas, if $\rho_U>-1$, $U$ is not integrable on $\Rset^+$.\\
Note that in the case $\rho_U=-1$, we can find examples of functions $U$ which are integrable or not.
\item[(vi)]
\emph{Sufficient condition for $U$ to belong to $\M$:}
Let $U$ be a positive and measurable function  with support $\Rset^+$, bounded on finite intervals. Then
$$
-\infty < \lim_{x\rightarrow\infty} \frac{\log\left(U(x)\right)}{\log(x)} <\infty \quad  \Longrightarrow \quad  U\in \mathcal{M}
$$
\end{itemize}
\end{propt}

To simplify the notation, when no confusion is possible, we will denote $\rho_U$ by $\rho$.

\begin{rmq} {\it Link to the notion of stochastic dominance}

Let $X$ and $Y$ be rv's with distributions $F_X$ and $F_Y$, respectively, with support $\Rset^+$.
We say that $X$ is smaller than $Y$ in the usual stochastic order (see e.g. \cite{shaked})  if
\begin{equation}\label{eq:20140926:001}
\barF_X(x)\leq\barF_Y(x)\quad\textrm{for all $x\in\Rset^+$.}
\end{equation}
This relation is also interpreted as the first-order stochastic dominance of $X$ over $Y$, as $F_X\ge F_Y$ (see e.g. \cite{HadarRussell1971}).

Let $X$, $Y$ be rv's such that $\barF_X=U$ and $\barF_Y=V$, where $U, V\in \M$ and  $\rho_U>\rho_V$. Then Properties~\ref{propt:main:001}, (ii), implies that there exists $x_0>0$ such that, for any $x\geq x_0$, $V(x)< U(x)$, hence that \eqref{eq:20140926:001} is satisfied at infinity, i.e.  that $X$ strictly dominates $Y$ at infinity.

Furthermore, the previous proof shows that a relation like \eqref{eq:20140926:001} is satisfied at infinity for any functions $U$ and $V$ in $\M$ satisfying $\rho_U>\rho_V$.
It means that the notion of first-order stochastic dominance or stochastic order confined to rv's can be extended to functions in $\M$.
In this way, we can say that if $\rho_U>\rho_V$, then $U$ strictly dominates $V$ at infinity.
\end{rmq}


Now let us define, for any positive and measurable function $U$ with support $\Rset^+$,
\begin{equation}\label{eq:main:000c}
\kappa_U:=\sup\left\{r:r\in\Rset\textrm{ \ \ and \ \ }\int_1^{\infty}x^{r-1}U(x)dx<\infty\right\}\textrm{.}
\end{equation}
Note that $\kappa_U$ may take values $\pm \infty$.

\begin{defi}\label{eq:main:defi:kappa}
For $U\in\mathcal{M}$, $\kappa_U$ defined in \eqref{eq:main:000c} is called the $\mathcal{M}$-index of $U$.
\end{defi}

\begin{rmq}\label{rmq:20140725:001}~
\begin{enumerate}
\item
If the function $U$ considered in \eqref{eq:main:000c} is bounded on finite intervals, then the integral involved can be computed on any interval $[a,\infty)$ with $a>1$.

\item When assuming $U=\overline{F}$, $F$ being a continuous distribution, the integral in (\ref{eq:main:000c}) reduces (by changing the order of integration), for $r>0$, 
to an expression of moment of a  rv: 
$$
\int_1^{\infty}x^{r-1}\overline{F}(x)dx=\frac{1}{r}\int_1^{\infty}\left(x^r-1\right)dF(x)=\frac{1}{r}\int_1^{\infty}x^rdF(x)-\frac{\barF(1)}{r}\textrm{.}
$$
\item We have $\kappa_U \geq 0$ for any tail $U=\overline{F}$ of a distribution $F$.\\

Indeed, suppose there exists $\overline{F}$ such that $\kappa_{\overline F}<0$. Let us denote $\kappa_{\overline F}$ by $\kappa$. 
Since $\kappa<\kappa/2<0$, we have by definition of $\kappa$ that
$\displaystyle \int_1^{\infty}x^{\kappa/2-1}\overline{F}(x)dx=\infty $.
But, since $\overline F\le 1$ and $\kappa/2-1<-1$, we can also write that
$\displaystyle \int_1^{\infty}x^{\kappa/2-1}\overline{F}(x)dx\leq\int_1^{\infty}x^{\kappa/2-1}dx<\infty $. 
Hence the contradiction. 

\item
A similar statement to Properties \ref{propt:main:001}, (iii), has been proved for RV functions (see \cite{BinghamGoldieTeugels}).
\end{enumerate}
\end{rmq}

Let us develop a simple example, also useful for the proofs.

\begin{exm}\label{lem:main:003}
Let $\alpha\in\Rset$ and $U_\alpha$ the function defined on $(0,\infty)$ by
$$
U_\alpha(x):=\left\{
\begin{array}{lcl}
1\text{,} &  & 0< x<1 \\
x^\alpha\text{,} &  & x\geq1\,.
\end{array}
\right.
$$

Then $U_\alpha\in\mathcal{M}$ with $\rho_{U_\alpha}=\alpha$ defined in (\ref{Mkappa}), and its $\mathcal{M}$-index satisfies $\kappa_{U_\alpha}=-\alpha$.
\end{exm}

To check that $U_\alpha\in\mathcal{M}$, it is enough to find a $\rho_{U_\alpha}$, since its unicity follows by Properties \ref{propt:main:001}, (i).
Choosing  $\rho_{U_\alpha}=\alpha$, we obtain, for any $\epsilon>0$, that
$$
\lim_{x\rightarrow\infty}\frac{U_\alpha(x)}{x^{\rho_{U_\alpha}+\epsilon}}=
\lim_{x\rightarrow\infty}\frac{1}{x^{\epsilon}}=0
\quad\text{and} \quad
\lim_{x\rightarrow\infty}\frac{U_\alpha(x)}{x^{\rho_{U_\alpha}-\epsilon}}=
\lim_{x\rightarrow\infty}x^{\epsilon}=\infty
$$
Hence $U_\alpha$ satisfies (\ref{Mkappa}) with $\rho_{U_\alpha}=\alpha$.\\
Now, noticing that
$$
\int_1^{\infty}x^{s-1}U_\alpha(x)dx=\int_1^{\infty}x^{s+\alpha-1}dx <\infty 
\quad \Longleftrightarrow \quad s+\alpha<0 
$$
then it comes that  $\kappa_{U_\alpha}$ defined  in \eqref{eq:main:000c} satisfies $\kappa_{U_\alpha}=-\alpha$.\hfill $\Box$

As a consequence of the definition of the $\M$-index $\kappa$ on $\mathcal{M}$, we can prove that Properties \ref{propt:main:001}, (vi), is not only a sufficient but also a necessary condition,
obtaining then a first characterization of $\mathcal{M}$.
\begin{teo}\label{teo:main:001} \textbf{First characterization of $\mathcal{M}$} \\
Let $U$ be a positive measurable function with support $\Rset^+$ and bounded on finite intervals. Then 
\begin{equation}\label{eq:main:000b}
U\in\mathcal{M} \,\text{with} \, \rho_U= - \tau \quad\Longleftrightarrow \quad\lim_{x\rightarrow\infty} \frac{\log\left(U(x)\right)}{\log(x)}= - \tau
\end{equation}
where $\rho_U$ is defined in \eqref{Mkappa}.
\end{teo}

\begin{exm}
The function $U$ defined by $U(x)=x^{\sin(x)}$ does not belong to $\M$ since the limit expressed in (\ref{eq:main:000b}) does not  exist . 
\end{exm}

Other properties on $\M$ can be deduced from Theorem \ref{teo:main:001}, namely:

\begin{propt}\label{propt:20140630}
Let $U$, $V$ $\in$ $\M$ with $\rho_U$ and $\rho_V$ defined in \eqref{Mkappa}, respectively. Then:
\begin{itemize}
\item[(i)]
The product \ $U\,V\in\M$ with $\rho_{U\,V}=\rho_U+\rho_V$.
\item[(ii)]
If $\rho_U\leq\rho_V<-1$ or $\rho_U<-1<0\leq\rho_V$, then the convolution $U\ast V\in\M$ with $\rho_{U\ast V}=\rho_V$.
If $-1<\rho_U\leq\rho_V$, then $U\ast V\in\M$ with $\rho_{U\ast V}=\rho_U+\rho_V$+1.
\item[(iii)]
If $\displaystyle \limx V(x)=\infty$, then $U\circ V\in\M$ with $\rho_{U\circ V}=\rho_U\,\rho_V$.
\end{itemize}
\end{propt}

\begin{rmq}\label{rmq:20140724:001}~
A similar statement to Properties \ref{propt:20140630}, (ii), has been proved when restricting the functions $U$ and $V$ to RV probability density functions, showing first $\displaystyle \limx\frac{U\ast V(x)}{U(x)+V(x)}=1$ (see \cite{BinghamGoldieOmey2006}).
In contrast, we propose a direct proof, under  the condition of integrability of  the function of $\M$ having the lowest $\rho$.

When $U$ and $V$ are tails of distributions belonging to RV, with the same tail index, Feller (\cite{feller21966}) proved that the convolution of $U$ and $V$ also belongs to this class and has the same tail index as $U$ and $V$.

\end{rmq} 

We can give a second way to characterize $\mathcal{M}$ using $\kappa_U$ defined in \eqref{eq:main:000c}.

\begin{teo}\label{teo:main:002}
\textbf{Second characterization of $\mathcal{M}$} \\
Let $U$ be a positive measurable function with support $\Rset^+$, bounded on finite intervals. Then
\begin{eqnarray}
U\in\mathcal{M_\fty} \, \text{with associated\ } \, \rho_U &\Longleftrightarrow  & \kappa_U=-\rho_U \quad\label{teo2a}
\end{eqnarray}
where $ \rho_U$ satisfies \eqref{Mkappa} and $\kappa_U$ satisfies \eqref{eq:main:000c}.
\end{teo}

Here is another characterization of $\mathcal{M}_\fty$, of Karamata type.

%
\begin{teo}\label{teo:main:003}
\textbf{Representation Theorem of Karamata type for $\mathcal{M}_\fty$}~
\begin{itemize}
\item[(i)]
Let $U\in\mathcal{M}_\fty$ with finite $\rho_U$ defined in (\ref{Mkappa}). There exist $b>1$ and functions
$\alpha$, $\beta$ and $\epsilon$ satisfying, as $x\rightarrow\infty$,
\begin{equation}\label{alpha-eps-beta}
\alpha(x)/\log(x) \,\, \rightarrow \,\, 0 \, , \qquad  \epsilon(x) \,\, \rightarrow \,\, 1 \, , \qquad 
\beta(x) \,\, \rightarrow \,\, \rho_U,
 \end{equation}
such that, for $ x\geq b$,
\begin{equation}\label{eq:main:000d}
U(x)=\exp\left\{\alpha(x)+\epsilon(x)\,\int_b^x\frac{\beta(t)}{t}dt\right\}\textrm{.}
\end{equation}
\item[(ii)]
Conversely, if there exists a positive measurable function $U$ with support $\Rset^+$, bounded on finite intervals,
satisfying (\ref{eq:main:000d}) for some $b>1$ and functions $\alpha$, $\beta$, and $\epsilon$ satisfying \eqref{alpha-eps-beta},
then $U\in\mathcal{M}_\fty$ with finite $\rho_U$ defined in (\ref{Mkappa}).
\end{itemize}
\end{teo}

\begin{rmq}~
\begin{enumerate}
\item Another way to express \eqref{eq:main:000d} is the following:
\begin{equation}\label{eq:teo:main:003:004}
U(x)=\exp\left\{\alpha(x)+\frac{\epsilon(x)\,\log(x)}{x}\,\int_b^x\beta(t)dt\right\}\textrm{.}
\end{equation}
\item The function $\alpha$ defined in Theorem \ref{teo:main:003} is not necessarily bounded, contrarily to the case of Karamata representation for RV functions.
\end{enumerate}
\end{rmq}

\begin{exm}\label{exm:main:intro:002}
Let $U\in\M$ with $\M$-index $\kappa_U$. If there exists $c>0$ such that $U<c$, then $\kappa_U\geq0$.
\end{exm}
Indeed, since we have
$\displaystyle 
\lim_{x\rightarrow\infty}\frac{\log\left(1/U(x)\right)}{\log(x)}\geq
\lim_{x\rightarrow\infty}\frac{\log\left(1/c\right)}{\log(x)}=0\textrm{,}
$
applying Theorem \ref{teo:main:001} allows to conclude.\hfill $\Box$

\subsection{Extension of the class $\mathcal{M}$}

We extend the class $\mathcal{M}$ introducing two other classes of functions.

\begin{defi}\label{eq:main:defi:classMextension}
$\mathcal{M}_\infty$ and $\mathcal{M}_{-\infty}$ are the classes of positive measurable functions $U$ with support $\Rset^+$, bounded on finite intervals, defined as
\begin{equation}\label{M+}
\M_\infty :=\left\{ U ~ : ~ \forall \rho \in\Rset, \, \lim_{x\rightarrow\infty}\frac{U(x)}{x^{\rho}}=0\right\}
\end{equation}
and
\begin{equation}\label{M-}
\M_{-\infty} :=\left\{ U ~ : ~ \forall\rho\in\Rset, \, \lim_{x\rightarrow\infty}\frac{U(x)}{x^{\rho}}=\infty\right\}
\end{equation}
\end{defi}

Notice that it would be enough to  consider $\rho<0$ ($\rho>0$, respectively) in (\ref{M+}) ((\ref{M-}), respectively),  and that $\M_\infty$, $\M_{-\infty}$ and $\M_\fty$ are disjoint. 

We denote by $\M_{\pm\infty}$ the union $\M_\infty\cup\M_{-\infty}$.

We obtain similar properties for $\M_\infty$ and $\M_{-\infty}$, as the ones given for $\M$, namely:
\begin{propt}\label{propt:main:001extension}~
\vspace{-2mm}
\begin{itemize}
\item[(i)]  $U\in\M_\infty \quad \Longleftrightarrow \quad 1/U\in\M_{-\infty}$.
\item[(ii)]
If
$
\quad(U,V)\textrm{\  $\in$\ }\M_{-\infty}\times\M\textrm{\  or\ }\M_{-\infty}\times\M_\infty\textrm{\  or\ }\M\times\M_\infty\textrm{,}
$
then $\displaystyle \, \lim_{x\rightarrow\infty}\frac{V(x)}{U(x)}=0$.
\item[(iii)]
If $U,V\in\M_{\infty}$ ($\M_{-\infty}$ respectively), then $U+V\in\M_{\infty}$ ($\M_{-\infty}$ respectively).
\end{itemize}
\end{propt}

The index $\kappa_U$ defined in (\ref{eq:main:000c}) may also be used to analyze $\M_\infty$ and $\M_{-\infty}$.
It can take infinite values, as can be seen in the following example.

\begin{exm}
Consider $U$ defined on $\Rset^+$ by \ $U(x):=e^{-x}$. Then $U\in\M_\infty$ with $\kappa_U=\infty$. 
Choosing $U(x)=e^{x}$ leads to $U\in\M_{-\infty}$ with $\kappa_U=-\infty$.
\end{exm}

A first characterization of $\M_\infty$ and $\M_{-\infty}$ can be provided, as done for $\M$ in Theorem~\ref{teo:main:001}. 
\begin{teo}\label{teo:main:001extension}
\textbf{First characterization of $\M_\infty$ and $\M_{-\infty}$} \\
Let $U$ be a positive measurable function with support $\Rset^+$, bounded on finite intervals. Then we have
\begin{equation}\label{eq:main:000bextension01}
U\in\M_\infty \quad \Longleftrightarrow  \quad\lim_{x\rightarrow\infty} \frac{\log\left(U(x)\right)}{\log(x)}= - \infty
\end{equation}
and
\begin{equation}\label{eq:main:000bextension02}
U\in\M_{-\infty} \quad \Longleftrightarrow \quad\lim_{x\rightarrow\infty} \frac{\log\left(U(x)\right)}{\log(x)}= \infty .
\end{equation}
\end{teo}

\begin{rmq}\label{rmq:20140920:001} {\it Link to a result from Daley and Goldie}. 

If we restrict $\M\cup\M_{\pm\infty}$ to tails of distributions, then combining Theorems \ref{teo:main:001} and \ref{teo:main:001extension} and Theorem 2 in \cite{DaleyGoldie2006}
provide another characterization, namely
$$
U\in\M\cup\M_{\pm\infty}\quad\Longleftrightarrow\quad X_U\in\M^{DG}
$$
where $X_U$ is a rv with tail $U$ and $\M^{DG}$ is the set of non-negative rv's $X$ having the property introduced by Daley and Goldie (see \cite{DaleyGoldie2006}) that
$$
\kappa(X\wedge Y)=\kappa(X)+\kappa(Y)
$$
for independent rv's $X$ and $Y$.
We notice that $\kappa(X)$ defined in \cite{DaleyGoldie2006} (called there the moment index) and applied to rv's, coincides with the $\M$-index of $U$,  when $U$ is the tail of the distribution of $X$.
\end{rmq}

An application of Theorem \ref{teo:main:001extension} provides similar properties as Properties \ref{propt:20140630}, namely:
\begin{propt}\label{propt:20140630ext}~
\vspace{-2mm}
\begin{itemize}
\item[(i)]
Let $(U,V)$ $\in$ $\M_{\infty}\times\M_\infty$ or $\M_{\pm\infty}\times\M$ or $\M_{-\infty}\times\M_{-\infty}$. Then
$U\cdot V\in\M_\infty$ or $\M_{\pm\infty}$ or $\M_{-\infty}$, respectively.
\item[(ii)]
Let $(U,V)$ $\in$ $\M_{\infty}\times\M$ with $\rho_V\geq0$ or $\rho_V<-1$, then $U\ast V\in\M$ with $\rho_{U\ast V}=\rho_V$.
Let $(U,V)\in\M_\infty\times\M_\infty$, then $U\ast V\in\M_\infty$.
Let $(U,V)$ $\in$ $\M_{-\infty}\times\M$ or $\M_{-\infty}\times\M_{\pm\infty}$, then $U\ast V\in\M_{-\infty}$.
\item[(iii)]
Let $U$ $\in$ $\M_{\pm\infty}$ and $V$ $\in$ $\M$ such that $\displaystyle \limx V(x)=\infty$ or $V$ $\in$ $\M_{-\infty}$, then
$U\circ V\in\M_{\pm\infty}$.
\end{itemize}
\end{propt}

Looking for extending Theorems~\ref{teo:main:002}-\ref{teo:main:003} to $\M_\infty$ and $\M_{-\infty}$ provide the next results.
On the converses of the result corresponding to Theorem~\ref{teo:main:002} extra-conditions are required.

\begin{teo}\label{teo:main:002extension}~

Let $U$ be a positive measurable function with support $\Rset^+$, bounded on finite intervals, with $\kappa_U$ defined in (\ref{eq:main:000c}). Then
\vspace{-2mm}
\begin{itemize}
\item[(i)]
\begin{itemize}
\item[(a)]
$  U\in\M_\infty\quad \Longrightarrow \quad \kappa_U=\infty$.
\item[(b)]
$U$ continuous, $\displaystyle \limx U(x)=0$, and $\kappa_U=\infty\quad \Longrightarrow \quad  U\in\M_\infty$.
\end{itemize}
\item[(ii)]
\begin{itemize}
\item[(a)]
$U\in\M_{-\infty}  \quad \Longrightarrow \quad \kappa_U=-\infty$.
\item[(b)]
$U$ continuous and non-decreasing, and $\kappa_U=-\infty\quad \Longrightarrow \quad  U\in\M_{-\infty}$.
\end{itemize}
\end{itemize}
\end{teo}
\begin{rmq}\label{rmq:20141107:001}~
\vspace{-2mm}
\begin{enumerate}
\item 
In (i)-(b), the condition $\kappa_U=\infty$ might appear intuitively sufficient to prove that $U\in\M_{\infty}$.
This is not true, as we can see with the following example showing for instance that the continuity assumption is needed. 
Indeed,  we can check that the function $U$ defined on $\Rset^+$ by
$$
U(x):=\left\{
\begin{array}{ll}
1/x & \mbox{if} \quad  x\in \underset{n\in \Nset\backslash \{0\}}{\bigcup} (n; n+1/n^n) \\    
e^{-x} & \mbox{otherwise},
\end{array}
\right.
$$
satisfies $\kappa_U=\infty$ and $\displaystyle \limx U(x)=0$, but is not continuous and does not belong to $\M_{\infty}$.
\item
The proof of (i)-(b) is based on an integration by parts, isolating the term $t^rU(t)$. The continuity of $U$ is needed, otherwise we would end up with an infinite number of jumps of the type 
$U(t^+)-U(t^-)(\neq 0)$ on $\Rset^+$.
\end{enumerate}
\end{rmq}

\begin{teo}\label{teo:main:003extension}
\textbf{Representation Theorem of Karamata Type for $\M_\infty$ and $\M_{-\infty}$}
\begin{itemize}
\item[(i)] If $U\in\M_{\infty}$, then there exist $b>1$ and a positive measurable function
$\alpha$ satisfying
\begin{equation}\label{eq:main:alpha}
\textrm{$\displaystyle \alpha(x)/\log(x) \, \underset{x\to\infty}{\rightarrow} \, \infty$,}
\end{equation}
such that, $\forall x\ge b$,
\begin{equation}\label{eq:main:000dextension01}
U(x)=\exp\left\{-\alpha(x)\right\}.
\end{equation}
\item[(ii)] If $U\in\M_{-\infty}$, then there exist $b>1$ and a positive measurable function
$\alpha$ satisfying (\ref{eq:main:alpha}) such that, $\forall x\ge b$,
\begin{equation}\label{eq:main:000dextension02}
U(x)=\exp\left\{\alpha(x)\right\}.
\end{equation}
\item[(iii)]
Conversely, if there exists a positive function $U$ with support $\Rset^+$, bound\-ed on finite intervals,
satisfying (\ref{eq:main:000dextension01}) or (\ref{eq:main:000dextension02}), respectively, for some positive function $\alpha$ satisfying (\ref{eq:main:alpha}),
then $U\in\mathcal{M}_\infty$ or $U\in\mathcal{M}_{-\infty}$, respectively.
\end{itemize}
\end{teo}

\subsection{On the complement set of $\M\cup\M_{\pm\infty}$}

Considering measurable functions $U:\Rset^+\to\Rset^+$, we have, applying Theorems \ref{teo:main:001} and \ref{teo:main:001extension}, that $U$ belongs to $\M$, $\M_\infty$ or $\M_{-\infty}$ if and only if $\displaystyle \limx \frac{\log\left(U(x)\right)}{\log(x)}$ exists, finite or infinite.
Using the notions (see for instance \cite{BinghamGoldieTeugels}) of \emph{lower order} of $U$, defined by
\begin{equation}\label{eq:20140609:001}
\mu(U):=\liminfx\frac{\log\left(U(x)\right)}{\log(x)}\textrm{,}
\end{equation}
and \emph{upper order} of $U$, defined by
\begin{equation}\label{eq:20140609:001bis}
\nu(U):=\limsupx\frac{\log\left(U(x)\right)}{\log(x)}\textrm{,}
\end{equation}
we can rewrite this characterization simply by $\mu(U)=\nu(U)$.

Hence, the complement set of $\M\cup\M_{\pm\infty}$ in the set of functions $U:\Rset^+\to\Rset^+$, denoted by $\Oset$, can be written as 
$$
\Oset:= \{U:\Rset^+\rightarrow\Rset^+~: \mu(U)<\nu(U)\}.
$$
This set is nonempty: $\Oset\neq\emptyset$, as we are going to see through examples.
A natural question is whether the \PBdH\ (see Theorem \ref{tpbdh} in Appendix \ref{ProofsofresultsconcerningOset}) applies when restricting $\Oset$ to tails of distributions.
The answer follows.
\begin{teo}\label{teo:20140801:001}~\\
Any distribution of a rv having a tail in $\Oset$ does not satisfy \PBdH.
\end{teo}

Examples of distributions $F$ satisfying $\mu(\barF)<\nu(\barF)$ are not well-known.
A non explicit one was given by Daley (see \cite{Daley2001}) when considering rv's with discrete support (see \cite{DaleyGoldie2006}).
We will provide a couple of explicit parametrized examples of functions in $\Oset$ which include tails of distributions with discrete support.
These functions can be extended easily to continuous positive functions not necessarily monotone, for instance adapting polynomes given by Karamata (see \cite{Karamata1931a001}).
These examples are more detailed in Appendix \ref{ProofsofresultsconcerningOset}.

\begin{exm}\label{exm:20140802:001}~

Let $\alpha>0$, $\beta\in\Rset$ such that $\beta\neq-1$, and $x_a>1$.
Let us consider the increasing series defined by $x_n=x_a^{(1+\alpha)^n}$, $n\geq1$, well-defined because $x_a>1$.
Note that 
$x_n\rightarrow\infty$ as $n\rightarrow\infty$.

The function $U$ defined by
\begin{equation}\label{eq:20140610:003}
U(x):=\left\{
\begin{array}{ll}
1\textrm{,} & 0\leq x<x_1 \\
x_n^{\alpha(1+\beta)}\textrm{,} & x\in[x_n,x_{n+1})\textrm{,} \quad \forall n\geq1
\end{array}
\right.
\end{equation}
belongs to $\Oset$, with
$$
\left\{
\begin{array}{cc}
\displaystyle\mu(U)=\frac{\alpha(1+\beta)}{1+\alpha}\quad\textrm{and}\quad\nu(U)=\alpha(1+\beta)\textrm{,} & \textrm{if $1+\beta>0$} \\
 & \\
\displaystyle\mu(U)=\alpha(1+\beta)\quad\textrm{and}\quad\nu(U)=\frac{\alpha(1+\beta)}{1+\alpha}\textrm{,} & \textrm{if $1+\beta<0$.}
\end{array}
\right.
$$
Moreover, if $1+\beta<0$, then $U$ is a tail of distribution whose associated rv has moments lower than $-\alpha(1+\beta)\big/(1+\alpha)$.
\end{exm}

\begin{exm}\label{exm:20140802:002}~

Let $c>0$ and $\alpha\in\Rset$ such that $\alpha\neq0$.
Let $(x_n)_{n\in\Nset}$ be defined by $x_1=1$ and $x_{n+1}=2^{x_n/c}$, $n\geq1$, well-defined for $c>0$.
Note that $x_n\to\infty$ as $n\to\infty$.

The function $U$ defined by
$$
U(x):=\left\{
\begin{array}{ll}
1 & 0\leq x<x_1 \\
2^{\alpha x_n} & x_n\leq x<x_{n+1}\textrm{,} \quad \forall n\geq1
\end{array}
\right.
$$
belongs to $\Oset$, with
$$
\left\{
\begin{array}{cc}
\displaystyle\mu(U)=\alpha c\quad\textrm{and}\quad\nu(U)=\infty\textrm{,} & \textrm{if $\alpha>0$} \\
 & \\
\displaystyle\mu(U)=-\infty\quad\textrm{and}\quad\nu(U)=\alpha c\textrm{,} & \textrm{if $\alpha<0$.}
\end{array}
\right.
$$
Moreover, if $\alpha<0$, then $U$ is a tail of distribution whose associated rv has moments lower than $-\alpha c$.
\end{exm}

\section{Extension of RV results}

In this section, well-known results and fundamental in Extreme Value Theory, as Karamata's relations 
and Karamata's Tauberian Theorem, are discussed on $\mathcal{M}$. A key tool for the extension of these standard results to $\mathcal{M}$ is
the characterizations of $\M$ given in Theorems \ref{teo:main:001} and \ref{teo:main:002}.

First notice the relation between the class $\M$ introduced in the previous section and the class RV defined in (\ref{eq:000rv}).
\begin{prop}\label{prop:20140329:strictsubset}~
RV$_\rho$ ($\rho\in\Rset$) is a strict subset of $\M$.
\end{prop}

The proof of this claim comes from the Karamata relation (see \cite{Karamata1933}) given, for all RV function $U$ with index $\rho\in\Rset$, by
\begin{equation}\label{eq:kar:001:karamata}
\lim_{x\rightarrow\infty}\frac{\log\left(U(x)\right)}{\log(x)}=\rho\textrm{,}
\end{equation}
which implies, using Properties \ref{propt:main:001}, (vi), that $U\in\M$ with $\M$-index $\kappa_U=-\rho$.
Moreover, RV $\neq$ $\M$, noticing that, for $t>0$,
$
\displaystyle \lim_{x\rightarrow\infty}\frac{U(tx)}{U(x)}
$
does not necessarily exist, whereas it does for a RV function $U$.
For instance the function defined on $\Rset^+$ by $U(x)=2+\sin(x)$, is not RV, but 
$
\displaystyle \lim_{x\rightarrow\infty}\frac{\log\left(U(x)\right)}{\log(x)}=0\textrm{,}
$
hence $U\in\M$.

\subsection{Karamata's Theorem}

We will focus on the well-known Karamata Theorem developed for RV (see \cite{Karamata1930} and e.g. \cite{feller21966,BinghamGoldieTeugels}) to analyze its extension to $\M$.
Let us recall it, borrowing the version given in \cite{deHaan}.
\begin{teo}\textbf{Karamata's Theorem (\cite{Karamata1930}; e.g. \cite{deHaan})}\label{teo:KaramatasTheorem}~

Suppose $U:\Rset^+\rightarrow\Rset^+$ is Lebesgue-summable on finite intervals.
Then
\begin{itemize}
\item[(K1)]
$$
U\in\textrm{RV}_\rho\textrm{, $\rho>-1$}
\quad\Longleftrightarrow\quad
\lim_{x\rightarrow\infty}\frac{xU(x)}{\int_0^xU(t)dt}=\rho+1>0\textrm{.}
$$
\item[(K2)]
$$
U\in\textrm{RV}_\rho\textrm{, $\rho<-1$} \quad \Longleftrightarrow \quad \lim_{x\rightarrow\infty}\frac{xU(x)}{\int_x^{\infty}U(t)dt}=-\rho-1>0\textrm{.}
$$
\item[(K3)]
\begin{itemize}
\item[(i)]
$\displaystyle U\in\textrm{RV}_{-1} \quad \Longrightarrow \quad  \lim_{x\rightarrow\infty}\frac{xU(x)}{\int_0^xU(t)dt}=0$.
\item[(ii)]
$\displaystyle U\in\textrm{RV}_{-1} \mbox{ and } \int_0^{\infty}U(t)dt<\infty \quad\Longrightarrow\quad \lim_{x\rightarrow\infty}\frac{xU(x)}{\int_x^{\infty}U(t)dt}=0$.
\end{itemize}
\end{itemize}
\end{teo}

\begin{rmq}
The converse of (K3), (i), is wrong in general. A counterexample can be given by the  Peter and Paul distribution which satisfies $\displaystyle \lim_{x\rightarrow\infty}\frac{xU(x)}{\int_x^{\infty}U(t)dt}=0$ but is not $\textrm{RV}_{-1}$. We will come back on that, in more details, in \ $\S$ \ref{subsection2.1.2}.
\end{rmq}

Theorem~\ref{teo:KaramatasTheorem} is based on the existence of certain limits. We can extend some of the results to $\M$, even when theses limits do not exist, replacing them by more general expressions.

\subsubsection{Karamata's Theorem on $\M$}

Let us introduce the following conditions, in order to state the generalization of the Karamata Theorem  to $\mathcal{M}$:
\begin{eqnarray*}
&(C1r)& \frac{x^rU(x)}{\int_b^xt^{r-1}U(t)dt}\in\M\textrm{ with $\M$-index 0, }\, i.e.\quad  \limx\left(\frac{\log\left(\int_b^xt^{r-1}U(t)dt\right)}{\log(x)}-\frac{\log\left(U(x)\right)}{\log(x)}\right)=r  \\
&(C2r)& \frac{x^rU(x)}{\int_x^\infty t^{r-1}U(t)dt}\in\M\textrm{ with $\M$-index 0,  }\, i.e. \quad \limx\left(\frac{\log\left(\int_x^{\infty}t^{r-1}U(t)dt\right)}{\log(x)}-\frac{\log\left(U(x)\right)}{\log(x)}\right)=r\\
&&
\end{eqnarray*}

\begin{teo}\label{prop:kar:001}~\textbf{Generalization of the Karamata Theorem  to $\mathcal{M}$} \hfill

Let $U:\Rset^+\rightarrow\Rset^+$ be a Lebesgue-summable on finite intervals, and $b>0$. We have, for $r\in\Rset$, 
\begin{itemize}
\item[(K1$^*$)] 
$$
U\in\M\textrm{ with $\M$-index $(-\rho)$ such that $\rho+r>0$}
\quad\Longleftrightarrow\quad
\left\{
\begin{array}{l}
\lim_{x\rightarrow\infty}\frac{\log\left(\int_b^xt^{r-1}U(t)dt\right)}{\log(x)}=\rho+r>0\\
~\\
U \mbox{ satisfies } (C1r)
\end{array}
\right.
$$
\item[(K2$^*$)]
$$
U\in\M\textrm{ with $\M$-index $(-\rho)$ such that $\rho+r < 0$}
\quad\Longleftrightarrow\quad
\left\{
\begin{array}{l}
\lim_{x\rightarrow\infty}\frac{\log\left(\int_x^{\infty}t^{r-1}U(t)dt\right)}{\log(x)}=\rho+r<0\textrm{}\\
~\\
U \mbox{ satisfies } (C2r)
\end{array}
\right.
$$
\item[(K3$^*$)]
$$
U\in\M\textrm{ with $\M$-index $(-\rho)$ such that $\rho+r = 0$}
\quad\Longleftrightarrow\quad
\left\{
\begin{array}{l}
\lim_{x\rightarrow\infty}\frac{\log\left(\int_b^xt^{r-1}U(t)dt\right)}{\log(x)}=\rho+r=0\\
~\\
U \mbox{ satisfies } (C1r)
\end{array}
\right.
$$
\end{itemize}
\end{teo}

This theorem provides then a fourth characterization of $\M$.

Note that  if $r=1$, we can assume $b\ge 0$, as in the original Karamata's Theorem.

\begin{rmq}~
\begin{enumerate}
\item
Note that (K3$^*$) provides an equivalence contrarily to (K3).

\item
Assuming that $U$ satisfies the conditions $(C2r)$ and
\begin{equation}\label{cdtionIntFini}
 \int_1^{\infty}t^{r} U(t)dt \, <\, \infty
\end{equation}
we can propose a characterization of $U\in\M$ with $\M$-index $(r+1)$, 
namely
$$
U\in\M\textrm{ with $\M$-index $(r+1)$}
\quad\Longleftrightarrow\quad
\lim_{x\rightarrow\infty} \frac{\log\left(\int_x^{\infty}t^{r} U(t)dt\right)}{\log(x)}=0\textrm{.}
$$
This is the generalization of (K3) in Theorem~\ref{teo:KaramatasTheorem}, providing not only a necessary condition but also a sufficient one for $U$ to belong to $\M$, under the conditions 
$(C2r)$ and \eqref{cdtionIntFini}.
\end{enumerate}
\end{rmq}

\subsubsection{Illustration using Peter and Paul distribution} \label{subsection2.1.2}

The Peter and Paul distribution is a typical example of a function which is not RV.
It is defined by (see e.g. \cite{Goldie1978}, \cite{EmbrechtsOmey1984}, \cite{embrechts1997} or \cite{Mikosch})
\begin{equation}\label{def-PP}
F(x):=1-\sum_{k\geq1\textrm{: }2^k>x}2^{-k}\textrm{,\quad $x>0$.}
\end{equation}

Let us illustrate the characterization theorems when applied on \pyp; we do it for instance for Theorems \ref{teo:main:001} and \,\ref{prop:kar:001}, proving that this distribution belongs to $\M$. 

\begin{prop}~

The \pyp\ does not belong to RV, but to $\M$ with $\M$-index $1$.
\end{prop}

This proposition can be proved using Theorem \ref{teo:main:001} or Theorem \ref{prop:kar:001}. To illustrate the application of these two theorems, we develop the proof here and not in the appendix. 
\begin{itemize}
\item[(i)]
\emph{Application of Theorem \ref{teo:main:001}}

For $x\in[2^n;2^{n+1})$ ($n\geq0$), we have, using \eqref{def-PP},
$
\displaystyle \overline{F}(x)=
\sum_{k\geq n+1}2^{-k}=
2^{-n}\textrm{,}
$
from which we deduce that $\displaystyle \frac{n}{n+1}\leq-\frac{\log\left(\overline{F}(x)\right)}{\log(x)}<1$, hence
$\displaystyle 
\lim_{x\rightarrow\infty}\frac{\log\left(\overline{F}(x)\right)}{\log(x)}=-1$, which by Theorem~\ref{teo:main:001} is equivalent to
$$
\overline{F}\in\mathcal{M} \quad\text{with} \quad \M-\text{index}\;1.
$$

\item[(ii)]
\emph{Application of Theorem \ref{prop:kar:001}}

Let us prove that
$$
\lim_{x\rightarrow\infty}\frac{\log\left(\int_b^x\overline{F}(t)dt\right)}{\log(x)}=0.
$$

Suppose $2^n\leq x<2^{n+1}$ and consider $a\in\Nset$ such that $a<n$. Choose w.l.o.g. $b=2^a$.

Then the \pyp \, \eqref{def-PP} satisfies 

$$
\int_b^x\overline{F}(t)dt=\sum_{k=a}^{n-1}\int_{2^k}^{2^{k+1}}\!\!\! \!\! \overline{F}(t)dt+\int_{2^n}^{x} \! \overline{F}(t)dt
=\sum_{k=a}^{n-1}2^{-k}(2^{k+1}-2^k)+(x-2^n)2^{-n}=n-a+x2^{-n}-1.
$$
Hence it comes
$$
\frac{\log(n-a+x2^{-n}-1)}{(n+1)\,\log(2)}\leq\frac{\log\left(\int_b^x\overline{F}(t)dt\right)}{\log(x)}\leq\frac{\log(n-a+x2^{-n}-1)}{n\,\log(2)}
$$

and, since $1\leq 2^{-n} x<2$, we obtain
$
\displaystyle \lim_{x\rightarrow\infty}\frac{\log\left(\int_b^x\overline{F}(t)dt\right)}{\log(x)}=0
$.\\
Moreover, we have 
$$
\lim_{x\rightarrow\infty}\frac{\log\left(\frac{x\barF(x)}{\int_b^x\overline{F}(t)dt}\right)}{\log(x)}=
1+\lim_{x\rightarrow\infty}\frac{\log\left(\barF(x)\right)}{\log(x)}-\lim_{x\rightarrow\infty}\frac{\log\left(\int_b^x\overline{F}(t)dt\right)}{\log(x)}=1.
$$
Theorem \ref{prop:kar:001} allows then to conclude that \;
$\displaystyle 
\textrm{$\overline{F}\in\mathcal{M}$ with $\mathcal{M}$-index $1$}
$. \hfill $\Box$
\end{itemize}

Note that 
the original Karamata Theorem (Theorem \ref{teo:KaramatasTheorem}) does not allow to prove that the \pyp \ is RV or not, since the converse of (i) in (K3) does not hold,
contrarily to Theorem \ref{prop:kar:001}.
Indeed, although we can prove that
$$
\lim_{x\rightarrow\infty}\frac{x\,\overline{F}(x)}{\int_b^x\overline{F}(t)dt}=\lim_{x,n\rightarrow\infty}\frac{x\,2^{-n}}{n-a+x2^{-n}-1}=0,
$$
Theorem \ref{teo:KaramatasTheorem} does not imply that $\overline{F}$ is $\textrm{RV}_{-1}$.

\subsection{Karamata's Tauberian Theorem}

Let us recall the well-known Karamata Tauberian Theorem which deals on Laplace-Stieltjes (L-S) transforms and RV functions.

The L-S transform of a positive, right continuous function $U$ with support $\Rset^+$ and with local bounded variation, is defined by

\begin{equation}\label{eq:20140328:001}
\widehat{U}(s)
:=\int_{(0; \infty)}e^{-xs}dU(x) \textrm{,\quad $s>0$.}
\end{equation}

\begin{teo}\label{teo:KaramatasTauberianTheorem}\textbf{Karamata's Tauberian Theorem (see \cite{Karamata1931})}\hfill

If $U$ is a non-decreasing right continuous function with support $\Rset^+$ and satisfying $U(0^+)=0$, with finite L-S transform $\widehat{U}$, then, for $\alpha>0$,
$$
U\in\textrm{RV}_{\alpha}\textrm{\ \ at infinity}\quad\Longleftrightarrow\quad\widehat{U}\in\textrm{RV}_{\alpha}\textrm{\ \ at $0^+$.}
$$
\end{teo}

Now we present the main result of this subsection which extends only partly the Karamata Tauberian Theorem to $\M$.

\begin{teo}\label{teo:KaramatasTauberianTheoremExtension}~ 

Let  $U$ be a continuous function with support $\Rset^+$ and  local bounded variation, satisfying $U(0^+)=0$.
Let $g$ be defined on $\Rset^+$ by $g(x)=1/x$.
Then, for any $\alpha>0$, 
\begin{itemize}
\item[(i)]
$\displaystyle U\in\M\textrm{\ \ with $\M$-index $(-\alpha)$}
\quad \Longrightarrow \quad 
\widehat{U}\circ g \in\M\textrm{\ \ with $\M$-index $(-\alpha)$}
$.
\item[(ii)]
$
\left\{
\begin{array}{l}
\widehat{U}\circ g \in\M\textrm{\ \ with $\M$-index $(-\alpha)$} \\
\textrm{and\, $\exists\, \eta\in [0;\alpha)$ \, : \,$x^{-\eta}U(x)$ concave}
\end{array}
\right.
\quad \Longrightarrow \quad 
U\in\M\textrm{\ \ with $\M$-index $(-\alpha)$}
$.
\end{itemize}
\end{teo}

\subsection{Results concerning domains of attraction}
\label{Section2.3}

Von Mises (see \cite{vonMises1936}) formulated some sufficient conditions to guarantee that the maximum of a sample of independent and identically distributed (iid) rv's with a same distribution, when normalized, converges to a non-degenerate limit distribution belonging to the class of extreme value distributions.
In this subsection we analyze these conditions on $\mathcal{M}$.

Before presenting the well-known von Mises' conditions, let us recall the theorem of the three limit types.

\begin{teo} (see for instance \cite{FisherTippett1928}, \cite{Gnedenko1943}) \label{teo:FT-G}~

Let $(X_n,n\in\Nset)$ be a sequence of iid rv's and $\displaystyle M_n:=\max_{1\le i\le n} X_i$.
If there exist constants $(a_n,n\in\Nset)$ and $(b_n,n\in\Nset)$ with $a_n>0$ and $b_n\in\Rset$ such that 
\begin{equation}\label{eq:20140727:010}
P\left(\frac{M_n-b_n}{a_n}\leq x\right)=F^n(a_n x+b_n)\underset{n\rightarrow\infty}{\rightarrow}G(x)
\end{equation}
with $G$ a non degenerate distribution function, then $G$ is one of the three following types:
\begin{eqnarray*}
\textrm{Gumbel} & : & \Lambda(x):=\exp\left\{e^{-x}\right\}\textrm{,\quad $x\in\Rset$} \\
\textrm{Fr\'{e}chet} & : & \Phi_{\alpha}(x):=\exp\left\{-x^{-\alpha}\right\}\textrm{,\quad $x\geq0$,\, for some $\alpha>0$} \\
\textrm{Weibull} & : & \Psi_{\alpha}(x):=\exp\left\{-(-x)^{-\alpha}\right\}\textrm{,\quad $x<0$, \, for some $\alpha<0$}
\end{eqnarray*}
\end{teo}

The set of distributions $F$ satisfying \eqref{eq:20140727:010} is called the domain of attraction of $G$ and denoted by $DA(G)$.

In what follows, we refer to the domains of attraction related to distributions with support $\Rset^+$ only, so the Fr\'{e}chet class and the subclass of the Gumbel class, denoted by $DA(\Lambda_\infty)$, consisting of distributions $F\in DA(\Lambda)$ with endpoint $x^*:=\sup\{x:F(x)>0\}=\infty$.

Now, let us recall the von Mises' conditions.
\begin{enumerate}
\item[(vM1)]
Suppose that $F$, continuous and differentiable, satisfies $F'>0$ for all $x\geq x_0$, for some $x_0>0$. If there exists $\alpha>0$, such that
\begin{equation*} 
\lim_{x\rightarrow\infty}\frac{x\,F'(x)}{\overline{F}(x)}=\alpha\textrm{,}
\end{equation*}
then $F\in DA(\Phi_{\alpha})$.
\item[(vM2)] Suppose that $F$ with infinite endpoint,  is continue and twice differentiable for all $x\geq x_0$, with $x_0>0$. If
\begin{equation*}
\lim_{x\rightarrow\infty}\left(\frac{\overline{F}(x)}{F'(x)}\right)'=0\textrm{,}
\end{equation*}
then $F\in DA(\Lambda_{\infty})$.
\item[(vM2bis)] Suppose that $F$ with finite endpoint $x^*$, is continue and twice differentiable for all $x\geq x_0$, with $x_0>0$. If
$$
\lim_{x\rightarrow x^*}\left(\frac{\overline{F}(x)}{F'(x)}\right)'=0\textrm{,}
$$
then $F\in DA(\Lambda)\setminus DA(\Lambda_{\infty})$.
\end{enumerate}

It is then straightforward to deduce from the conditions (vM1) and (vM2), the next results.
\begin{prop}\label{prop:main:vonmises:001}~

Let $F$ be a distribution.
\begin{itemize}
\item[(i)]
If $F$ 
satisfies 
$\displaystyle \lim_{x\rightarrow\infty}\frac{x\,F'(x)}{\overline{F}(x)}=\alpha>0$, then $\overline{F}\in\mathcal{M}$ with $\mathcal{M}$-index $1/\alpha$.
\item[(ii)]
If $F$ 
satisfies 
$\displaystyle \limx\left(\frac{\overline{F}(x)}{F'(x)}\right)'=0$, then $\overline{F}\in\M_\infty$.
\end{itemize}
\end{prop}

So the natural question is how to relate $\mathcal{M}$ or $\M_\infty$ to the domains of attraction $DA(\Phi_{\alpha})$ and $DA(\Lambda_{\infty})$.
To answer it, let us recall three results on those domains of attraction that will be needed.
\begin{teo}(see e.g. \cite{deHaanFerreira}, Theorem 1.2.1) \label{teo:mises:dehaanferreira0}~

Let $\alpha>0$. The distribution function $F\in DA(\Phi_{\alpha})$ if and only if $x^*=\sup\{x:F(x)<1\}=\infty$ and $\overline{F}\in\textrm{RV}_{-\alpha}$.
\end{teo}

\begin{cor} De Haan (1970) (see \cite{deHaan}, Corollary 2.5.3) \label{cor:mises:dehaan0}~

If $F\in DA(\Lambda_{\infty})$, then
$\displaystyle \lim_{x\rightarrow\infty} \frac{\log\left(\overline{F}(x)\right)}{\log(x)} = -\infty$.
\end{cor}

\begin{teo} Gnedenko (see \cite{Gnedenko1943}, Theorem 7)  \label{teo:mises:gnedenko:20140412:001}~

The distribution function $F\in DA(\Lambda_{\infty})$ if and only if there exists a continuous function $A$ such that
$A(x)\rightarrow0$ as $x\rightarrow\infty$ and, for all $x\in\Rset$,
\begin{equation}\label{eq:mises:20140412:010bis}
\lim_{z\rightarrow\infty}\frac{1-F(z\,(1+A(z)\,x))}{1-F(z)}=e^{-x}\textrm{}
\end{equation}
\end{teo}
De Haan (\cite{deHaan1971}) noticed that Gnedenko did not use the continuity of $A$ to prove this theorem. 

These results allow to formulate the next statement.
\begin{teo}\label{teo:20140411:001}~
\begin{itemize}
\item[(i)] $\forall\alpha>0$, $F\in DA(\Phi_{\alpha})\ \Longrightarrow\ \barF\in\M$ with $\M$-index ($- \alpha$), and the converse does not hold:
$$
\displaystyle \{F\in DA(\Phi_{\alpha})\textrm{, }\alpha>0\}\, \subsetneq \;  \{F : ~\overline{F}\in\mathcal{M}\}\textrm{.}
$$
\item[(ii)] $\displaystyle DA(\Lambda_{\infty})\; \subsetneq \;  \{F : ~\overline{F}\in\mathcal{M}_\infty\}$.
\end{itemize}
\end{teo}

Let us give some examples illustrating the strict subset inclusions.

\begin{exm}  {\it The Peter and Paul distribution.} 

To show that $DA(\Phi_{\alpha}) \neq  \{F : ~\overline{F}\in\mathcal{M}\textrm{ with $\M$-index ($- \alpha$)}\}$, $\alpha>0$, in (i), it is enough to notice that the Peter and Paul distribution does not belong to $DA(\Phi_1)$, but its associated tail of distribution belongs to $\mathcal{M}$.
\end{exm}

\begin{exm}\label{exm:20140925:001}  

To illustrate (ii), we consider the distribution $F$ defined in a left neighborhood of $\infty$ by
\begin{equation}\label{eq:20140526:002}
F(x):=1-\exp\left(-\lfloor x\rfloor\,\log(x)\right)\textrm{,}
\end{equation}
Then it is straightforward to see that $F \in \{F : ~\overline{F}\in\mathcal{M}_\infty\}$, by Theorem \ref{teo:main:003extension} and the fact that 
$\displaystyle
\lim_{x\rightarrow\infty} \frac{\lfloor x\rfloor\,\log(x)}{\log(x)} = \infty\, .
$

We can check that $F\not \in DA(\Lambda_{\infty})$.
The proof, by contradiction, is given in Appendix \ref{ProofsofSection2.3}.
\end{exm}

\begin{rmq}~

Lemma 2.4.3 in \cite{deHaan} says that if $F\in DA(\Lambda_{\infty})$, then a continuous and increasing distribution function $G$ satisfying
\begin{equation}\label{eq:20140526:001}
\limx\frac{\overline{F}(x)}{\overline{G}(x)}=1\textrm{,}
\end{equation}
exists. Is it possible to extend this result to $\M$? The answer is no. To see that, it is enough to consider Example~\ref{exm:20140925:001}  with $F\in\M \setminus DA(\Lambda_{\infty}) $ defined in \eqref{eq:20140526:002} to see that the De Haan's result does not hold.

Indeed, suppose that for $F$ defined in \eqref{eq:20140526:002}, there exists a continuous and increasing distribution function $G$ satisfying \eqref{eq:20140526:001}, which comes back to suppose that there exits a positive and continuous function $h$ such that $G(x)=1-\exp\left(-h(x)\,\log(x)\right)$ ($x>0$), in particular in a neighborhood of $\infty$. 
So \eqref{eq:20140526:001} may be rewritten as
$$
\limx\frac{\overline{F}(x)}{\overline{G}(x)}=
\limx\exp\left(-\left(\lfloor x\rfloor-h(x)\right)\,\log(x)\right)=
\limx x^{h(x)-\lfloor x\rfloor}=
1\textrm{}
$$
However, since $\lfloor x\rfloor$ cannot be approximated for any continuous function,  the previous limit does not hold.
\end{rmq}

\section{Conclusion}

We introduced a new class of positive functions with support $\Rset^+$, denoted by $\M$, strictly larger than the class of RV functions at infinity.
We  extended to $\M$ some well-known results given on RV class, which are crucial to study extreme events. These new tools allow to expand EVT beyond RV.
This class satisfies a number of algebraic properties and its members $U$ can be characterized by a unique real number, called the $\M$-index $\kappa_U$. 
Four characterizations of $\M$ were provided, one of them being the extension to $\M$ of the well-known Karamata's Theorem restricted to RV class.
Furthermore, the cases $\kappa_U=\infty$ and $\kappa_U=-\infty$ were analyzed and their corresponding classes, denoted by $\M_\infty$ and $\M_{-\infty}$ respectively, were identified and studied, as done for $\M$. The three sets $\M_{\infty}$, $\M_{-\infty}$ and $\M$ are disjoint.
Tails of distributions not belonging to $\M\cup\M_{\pm\infty}$ were proved not to satisfy \PBdH.
Explicit examples of such functions and their generalization were given.

Extensions to $\M$ of the Karamata Theorems were discussed in the second part of the paper.
Moreover, we proved that the sets of tails of distributions whose distributions belong to the domains of attraction of Fr\'echet and Gumbel (with distribution support $\Rset^+$), are strictly included in $\M$ and $\M_\infty$ 
, respectively.

Note that any result obtained here can be applied to functions with finite support, i.e. finite endpoint $x^*$, by using the change of variable $y=1/(x^*-x)$ for $x<x^*$.

After having addressed the probabilistic analysis of  $\M$, we will look  at its statistical one. An interesting question is how to build estimators of the $\M$-index,
which could be used on RV since $\textrm{RV}\subseteq\M$. A companion paper addressing this question is in progress.

Finally, we will  develop a multivariate version of $\M$,  to represent and describe relations among random variables: dependence structure, tail dependence, conditional independence, and asymptotic independence.

\section*{Acknowledgments} 
Meitner Cadena acknowledges the support of SWISS LIFE through its ESSEC research program on 'Consequences of the population ageing on the insurances loss'. Partial support from RARE-318984 (an FP7 Marie Curie IRSES Fellowship) is also kindly acknowledged.



\appendix

\section{Proofs of results given in Section 1}


\subsection{Proofs of results concerning $\M$}

\begin{proof}[Proof of Theorem\,\ref{teo:main:001}]

The sufficient condition given in Theorem\,\ref{teo:main:001} comes from Properties\,\ref{propt:main:001}, (vi). So it remains to prove its necessary condition, namely that 
\begin{equation} \label{teo:main:001:lemma}
\lim_{x\rightarrow\infty}-\frac{\log\left(U(x)\right)}{\log(x)}=-\rho_U
\end{equation}
for $U\in\mathcal{M}$ with finite $\rho_U$ defined in (\ref{Mkappa}).

Let $\epsilon>0$ and define $V$ by
$$
V(x)=\left\{
\begin{array}{ll}
1\text{,} & 0< x<1 \\
x^{\rho_U+\epsilon}\text{,} & x\geq1
\end{array}
\right.
$$
Applying Example \ref{lem:main:003} with $\alpha=\rho_U+\epsilon$ with $\epsilon>0$ implies that $\rho_V=\rho_U+\epsilon$, hence $\rho_V>\rho_U$.
Using Properties \ref{propt:main:001}, (ii), provides then that 
$$
\lim_{x\rightarrow\infty}\frac{U(x)}{V(x)}=\lim_{x\rightarrow\infty}\frac{U(x)}{x^{\rho_U+\epsilon}}=0\textrm{,}
$$
so, for $n\in\Nset^*$, there exists $x_0>1$ such for all $x\geq x_0$, 
$$ 
\frac{U(x)}{x^{\rho_U+\epsilon}}\leq\frac{1}{n}\, , \quad \textit{i.e.}\quad n\,U(x)\leq x^{\rho_U+\epsilon}\textrm{.}
$$
Applying the logarithm function to this last inequality and dividing it by $-\log(x)$, $x \geq x_0$, gives
\begin{eqnarray*}
-\frac{\log(n)}{\log(x)}-\frac{\log(U(x))}{\log(x)} & \geq & -\rho_U-\epsilon\textrm{,}
\end{eqnarray*}
hence
$$
-\frac{\log(U(x))}{\log(x)}\geq-\rho_U-\epsilon 
$$
and then
$$
\liminf_{x\rightarrow\infty}-\frac{\log(U(x))}{\log(x)}\geq-\rho_U-\epsilon\textrm{.}
$$

We consider now the function 
$$
W(x)=\left\{
\begin{array}{ll}
1\text{,} & 0< x<1 \\
x^{\rho_U-\epsilon}\text{,} & x\geq1\textrm{}
\end{array}
\right.
$$
with $\epsilon>0$
and proceed in the same way to obtain that, 
for any $\epsilon>0$, 
$\displaystyle \limsup_{x\rightarrow\infty}-\frac{\log(U(x))}{\log(x)}\leq-\rho_U+\epsilon$.
Hence, $\forall \epsilon>0$, we have
$$
-\rho_U-\epsilon\leq\liminf_{x\rightarrow\infty}-\frac{\log(U(x))}{\log(x)}\leq\limsup_{x\rightarrow\infty}-\frac{\log(U(x))}{\log(x)}\leq-\rho_U+\epsilon
$$
from which the result follows taking $\epsilon$ arbitrary.
\end{proof}

Now we introduce a lemma, on which the proof of Theorem \ref{teo:main:002} will be based.
\begin{lem}\label{lem:main:004}
Let $U\in\mathcal{M}$ with associated $\mathcal{M}$-index $\kappa_U$ defined in (\ref{eq:main:000c}).
Then necessarily $\kappa_U=-\rho_U$, where $\rho_U$ is defined in (\ref{Mkappa}).
\end{lem}

\begin{proof}[Proof of Lemma \ref{lem:main:004}] 

Let $U\in\mathcal{M}$ with $\mathcal{M}$-index $\kappa_U$ given in \eqref{eq:main:000c} and $\rho_U$ defined in \eqref{Mkappa}.
By Theorem \ref{teo:main:001}, we have
$
\displaystyle \limx\frac{\log(U(x))}{\log(x)}=\rho_U
$.

Hence, for all $\epsilon>0$ there exists $x_0>1$ such that, for $x\geq x_0$,
$
U(x)\leq x^{\rho_U+\epsilon}
$.

Multiplying this last inequality by $x^{r-1}$, $r\in\Rset$, and integrating it on $[x_0;\infty)$, we obtain

$$
\int_{x_0}^\infty x^{r-1}U(x)dx\leq\int_{x_0}^\infty x^{\rho_U+\epsilon+r-1}dx
$$

which is finite 
if $r<-\rho_U-\epsilon$.
Taking $\epsilon\downarrow0$ then the supremum on $r$ leads to $\kappa_U=-\rho_U$.
\end{proof}


\begin{proof}[Proof of Theorem \ref{teo:main:002}] ~

The necessary condition is proved by Lemma\,\ref{lem:main:004}.
The sufficient condition follows from the assumption that $\rho_U$ satisfies \eqref{Mkappa}.
\end{proof}

\begin{proof}[Proof of Theorem \ref{teo:main:003}]~

\begin{itemize}
\item
\emph{Proof of (i)}
 
For $U\in\mathcal{M}$, Theorems \ref{teo:main:001} and \ref{teo:main:002} give that
\begin{equation}\label{eq:main:teoRepresentation:000}
\lim_{x\rightarrow\infty}-\frac{\log(U(x))}{\log(x)}=-\rho_U=\kappa_U\quad\textrm{with $\rho_U$ defined in (\ref{Mkappa}) and $\kappa_U$ in (\ref{eq:main:000c}).}
\end{equation}

Introducing a function $\gamma$ such that
\begin{equation}\label{eq:main:teoRepresentation:000bis}
\lim_{x\rightarrow\infty} \gamma(x) = 0
\end{equation}
we can write, for some $b>1$, applying the L'H\^{o}pital's rule to the ratio,
\begin{equation}\label{eq:main:teoRepresentation:001}
\lim_{x\rightarrow\infty}\left(\gamma(x)+\frac{\int_b^x\frac{\log(U(t))}{\log(t)}\,\frac{dt}{t}}{\log(x)}\right)=
\lim_{x\rightarrow\infty}\frac{\log(U(x))}{\log(x)}=-\kappa_U\textrm{.}
\end{equation}

\begin{itemize}
\item[$\triangleright$]
Suppose $\kappa_U\neq0$. 
Then we deduce from (\ref{eq:main:teoRepresentation:000}) and (\ref{eq:main:teoRepresentation:001}), that
\begin{equation}\label{eq:teo:main:003:002}
\lim_{x\rightarrow\infty}\frac{\log(U(x))}{\gamma(x)\,\log(x)+\int_b^x\frac{\log(U(t))}{t\log(t)}\,dt}=1\textrm{}
\end{equation}
Hence, defining the function $\displaystyle \epsilon_U(x):=\frac{\log(U(x))}{\gamma(x)\,\log(x)+\int_b^x\frac{\log(U(t))}{t\log(t)}\,dt}$, for $x\ge b$, 
we can express $U$, for $x\ge b$, as
$$
 U(x) = \exp\left\{\alpha_U(x) + \epsilon_U(x)\,\int_b^x\frac{\beta_U(t)}{t}\,dt \right\}
 $$
 \begin{equation}\label{betaU-const}
 \text{where}  \quad \alpha_U(x):=\epsilon_U(x)\,\gamma(x)\,\log(x) \quad \text{and} \quad \beta_U(x):=\frac{\log(U(x))}{\log(x)}.\quad
 \end{equation}
It is then straightforward to
check that 
the functions $\alpha_U$, $\beta_U$ and $\epsilon_U$ satisfy the conditions given in Theorem\,\ref{teo:main:003}.
Indeed, 
by \eqref{eq:main:teoRepresentation:000bis} and \eqref{eq:teo:main:003:002}, 
$
\displaystyle \lim_{x\rightarrow\infty}\frac{\alpha_U(x)}{\log(x)}=\lim_{x\rightarrow\infty}\epsilon_U(x)\,\gamma(x)=0
$. 
Using \eqref{eq:main:teoRepresentation:000}, we obtain
$
\displaystyle \lim_{x\rightarrow\infty}\beta_U(x)=\lim_{x\rightarrow\infty}\frac{\log(U(x))}{\log(x)}=-\kappa_U=\rho_U\textrm{.}
$
Finally, by \eqref{eq:teo:main:003:002}, we have
$
\displaystyle \lim_{x\rightarrow\infty}\epsilon_U(x)=1\,.
$
\item[$\triangleright$] Now suppose $\kappa_U=0$.

We want to prove \eqref{eq:main:000d} for some functions $\alpha$, $\beta$, and $\epsilon$ satisfying \eqref{alpha-eps-beta}.

Notice that \eqref{eq:main:teoRepresentation:000} with $\kappa_U=0$ allows to write  that
$\displaystyle \lim_{x\rightarrow\infty}\frac{\log(x\,U(x))}{\log(x)}=1$.

So applying Theorem\,\ref{teo:main:001} to the function $V$ defined by $V(x)=xU(x)$, gives that $\displaystyle V \in\mathcal{M}$ with $\rho_V=-\kappa_V=1$. 
Since $\kappa_V \neq 0$, we can proceed in the same way as previously, and obtain a representation for $V$ of the form \eqref{eq:main:000d}, namely,  for $d>1$, $\forall x\ge d$, 
$$
V(x)=\exp\left\{\alpha_V(x)+\epsilon_V(x)\,\int_{d}^x\frac{\beta_V(t)}{t}\, dt \right\}
$$
where $\alpha_V$, $\beta_V$, $\epsilon_V$ satisfy the conditions of Theorem\,\ref{teo:main:003}
and $\displaystyle \beta_V=\frac{\log(V(x))}{\log(x)}$ (see \eqref{betaU-const}). 
Hence we have, for $x\ge d$,
\begin{eqnarray*}
U(x)&=& \frac{V(x)}{x}  =\exp\left\{-\log(x)+\alpha_V(x)+\epsilon_V(x)\,\int_{d}^x\frac{\log(t\;U(t))}{t\;\log(t)}\, dt \right\} \\
 & = & \exp\left\{\alpha_V(x)+(\epsilon_V(x)-1)\,\log(x)-\epsilon_V(x)\,\log(d)+\epsilon_V(x)\,\int_{d}^x\frac{\log(U(t))}{t\;\log(t)}dt\right\}\textrm{.}
\end{eqnarray*}
Noticing that
$\displaystyle \lim_{x\rightarrow\infty}\frac{\alpha_V(x)+(\epsilon_V(x)-1)\,\log(x)-\epsilon_V(x)\,\log(d)}{\log(x)}=0 $, 
we obtain that $U$ satisfies \eqref{eq:main:000d} when setting, for $x\ge d$, $\displaystyle \alpha_U(x):=\alpha_V(x)+(\epsilon_V(x)-1)\,\log(x)-\epsilon_V(x)\,\log(d)$, $\displaystyle \beta_U(x):=\frac{\log(U(x))}{\log(x)}$ and  $\displaystyle \epsilon_U:=\epsilon_V$.
\end{itemize}

\item \emph{Proof of (ii)}

Let $U$ be a positive function with support $\Rset^+$, bounded on finite intervals. Assume that $U$ can be expressed as 
 \eqref{eq:main:000d} for some functions $\alpha$, $\beta$, and $\epsilon$ satisfying (\ref{alpha-eps-beta}). 
 We are going to check the sufficient condition given in Properties \ref{propt:main:001}, (vi), to prove that $U\in\mathcal{M}$.

Since
$
\displaystyle \frac{\log(U(x))}{\log(x)}=\frac{\alpha(x)}{\log(x)}+\epsilon(x)\frac{\int_b^x\frac{\beta(t)}{t}dt}{\log(x)}
$
and that, via L'H\^{o}pital's rule,
$$
\lim_{x\rightarrow\infty}\frac{\int_b^x\frac{\beta(t)}{t}dt}{\log(x)}=\lim_{x\rightarrow\infty}\frac{\beta(x)/x}{1/x}=\lim_{x\rightarrow\infty}\beta(x)
$$
then using the limits of $\alpha$, $\beta$, and $\epsilon$ allows to conclude.
\end{itemize}
\end{proof}

\begin{proof}[Proof of Properties \ref{propt:main:001}] ~
\begin{itemize}
\item {\it Proof of (i)}

Let us prove this property by contradiction.

Suppose there exist $\rho$ and $\rho'$, with $\rho'<\rho$, both satisfying \eqref{Mkappa}, for $U\in\mathcal{M}$. 
Choosing $\epsilon=(\rho-\rho')/2$ in \eqref{Mkappa} gives 
$$
\lim_{x\rightarrow\infty}\frac{U(x)}{x^{\rho'+\epsilon}}=0\textrm{\ \ \ and \ \ }\limx\frac{U(x)}{x^{\rho-\epsilon}}=\limx\frac{U(x)}{x^{\rho'+\epsilon}}=\infty\textrm{,}
$$
hence the contradiction.
\item {\it Proof of (ii)}

Choosing $\epsilon=(\rho_U-\rho_V)/2$, we can write
$$
\frac{V(x)}{U(x)}=\frac{V(x)}{x^{\rho_V+\epsilon}}\,\frac{x^{\rho_V+\epsilon}}{U(x)}
=\frac{V(x)}{x^{\rho_V+\epsilon}}\,\left(\frac{U(x)}{x^{\rho_U-\epsilon}}\right)^{-1}
$$
from which we deduce (ii).

\item {\it Proof of (iii)}

Let $U,V\in\M$, $a>0$, $\epsilon>0$ and suppose w.l.o.g. that $\rho_U\leq\rho_V$.

Since $\rho_V-\rho_U>0$, writing
$
\displaystyle \frac{aU(x)}{x^{\rho_V\pm\epsilon}}
$ 
$
\displaystyle =\frac{a}{x^{\rho_V-\rho_U}}\frac{U(x)}{x^{\rho_U\pm\epsilon}}
$
gives
$
\displaystyle \limx\frac{aU(x)+V(x)}{x^{\rho_V+\epsilon}}=0
$
and \newline
$
\displaystyle \limx\frac{aU(x)+V(x)}{x^{\rho_V-\epsilon}}=\infty
$,
we conclude thus that $\rho_{aU+V}=\rho_U\vee\rho_V$.
\item {\it Proof of (iv)}

It is straightforward since \eqref{Mkappa} can be rewritten as
$$
\lim_{x\rightarrow\infty}\frac{1/U(x)}{x^{-\rho_U-\epsilon}}
=\infty\textrm{\ \ \ and \ \ }\lim_{x\rightarrow\infty}\frac{1/U(x)}{x^{-\rho_U+\epsilon}}
=0
\textrm{.}
$$

\item {\it Proof of (v)}

First, let us consider $U\in\M$ with $\rho_U<-1$.

Choosing $\epsilon_0=-(\rho_U+1)/2$ ($>0$) in \eqref{Mkappa} implies that there exist $C>0$ and $x_0>1$ such that, for $x\geq x_0$,
$
\displaystyle U(x)\leq C\,x^{\rho_U+\epsilon_0}=C\,x^{(\rho_U-1)/2}
$,
from which we deduce that
$$
\int_{x_0}^\infty U(x)\,dx<\infty\textrm{.}
$$
We conclude that
$\displaystyle \int_0^\infty U(x)\,dx<\infty$ because $U$ is bounded on finite intervals.

Now suppose that $\rho_U>-1$.

Choosing $\epsilon_0=(\rho_U+1)/2$ ($>0$) in \eqref{Mkappa} gives that for $C>0$ there exists $x_0>1$ such that, for $x\geq x_0$,
$
U(x)\geq C\,x^{(\rho_U-1)/2}
$ 
$\displaystyle \int_0^\infty U(x)\,dx\geq\int_{x_0}^\infty U(x)\,dx\geq\infty$.
\item {\it Proof of (vi)}

Assuming $\displaystyle -\infty<\lim_{x\rightarrow\infty}\frac{\log\left(U(x)\right)}{\log(x)}<\infty $, 
we want to prove that $U$ satisfies \eqref{Mkappa},
which implies that $U\in\M$.

So let us prove \eqref{Mkappa}.

Consider
$
\displaystyle \rho=\lim_{x\rightarrow\infty}\frac{\log\left(U(x)\right)}{\log(x)}\textrm{}
$
well defined under our assumption,  and from which we can deduce that,
\begin{equation*}
\forall \epsilon>0,  \exists x_0>1 \, \text{such that}, \,\forall x\geq x_0,\quad -\frac{\epsilon}{2}\leq\frac{\log\left(U(x)\right)}{\log(x)}-\rho\leq\frac{\epsilon}{2}\textrm{.}
\end{equation*}
Therefore we can write that, for $x\geq x_0$, on one hand,
$$
0\leq \frac{U(x)}{x^{\rho+\epsilon}} 
=\exp\left\{\left(\frac{\log\left(U(x)\right)}{\log(x)}-\rho-\epsilon\right)\log(x)\right\} 
 \le   \exp\left\{-\frac{\epsilon}{2}\log(x)\right\} \underset{x\rightarrow\infty} {\longrightarrow} 0
$$
and on the other hand, 
$$
\frac{U(x)}{x^{\rho-\epsilon}} =\exp\left\{\left(\frac{\log\left(U(x)\right)}{\log(x)}-\rho+\epsilon\right)\log(x)\right\}  \ge \exp\left\{\frac{\epsilon}{2}\log(x)\right\}\underset{x\rightarrow\infty} {\longrightarrow} \infty
$$
hence the result.
\end{itemize}
\end{proof}

\begin{proof}[Proof of Properties \ref{propt:20140630}] ~

Let $U$, $V$ $\in$ $\M$ with $\rho_U$ and $\rho_V$ respectively, defined in \eqref{Mkappa}.
\begin{itemize}
\item {\it Proof of (i)}~

It is immediate since
$$
\limx\frac{\log\left(U(x)\,V(x)\right)}{\log(x)}=
\limx\left(\frac{\log\left(U(x)\right)}{\log(x)}+\frac{\log\left(V(x)\right)}{\log(x)}\right)=\rho_U+\rho_V
$$
\item {\it Proof of (ii)}~

First notice that, since $U, V\in\M$, via Theorems \ref{teo:main:001} and \ref{teo:main:002},
for $\epsilon>0$, there exist $x_U>0$, $x_V>0$, such that, for $x\geq x_0=x_U\vee x_V$,
$$
x^{\rho_U-\epsilon/2}\leq U(x)\leq x^{\rho_U+\epsilon/2}
\quad\textrm{and}\quad
x^{\rho_V-\epsilon/2}\leq V(x)\leq x^{\rho_V+\epsilon/2}\textrm{.}
$$

\begin{itemize}
\item[$\triangleright$]
\emph{Assume $\rho_U \leq \rho_V < -1$}. Hence, via Properties \ref{propt:main:001}, (v), both $U$ and $V$ are integrable on $\Rset^+$.
Choose $\rho=\rho_V$.

Via the change of variable $s=x-t$, we have, $\forall$ $x\geq 2x_0>0$,
\begin{eqnarray*}
\lefteqn{\frac{U\ast V(x)}{x^{\rho+\epsilon}} = \int_0^{x/2} U(t)\frac{V(x-t)}{x^{\rho+\epsilon}}dt+\int_{x/2}^x U(t)\frac{V(x-t)}{x^{\rho+\epsilon}}dt} \\
 & & \le \frac{1}{x^{\epsilon/2}}\int_0^{x/2} U(t)\left(1-\frac{t}{x}\right)^{\rho_V+\epsilon/2}dt
+\frac{1}{x^{\rho_V-\rho_U+\epsilon/2}}\int_0^{x/2} V(s)\left(1-\frac{s}{x}\right)^{\rho_U+\epsilon/2}ds \\
& & \le \frac{\max\left(1,c^{\rho_V+\epsilon/2}\right)}{x^{\epsilon/2}}\int_0^{x/2} U(t)dt+\frac{\max\left(1,c^{\rho_U+\epsilon/2}\right)}{x^{\rho_V-\rho_U+\epsilon/2}}\int_0^{x/2}V(s)ds
\end{eqnarray*}
since, for $0\leq t \leq x/2$, {\it i.e.} $\displaystyle 0<c<\frac{1}{2}\leq1-\frac{t}{x}\leq1$, 
$$
\left(1-\frac{t}{x}\right)^{\rho_V+\epsilon/2}\!\!\! \leq \max\left(1,c^{\rho_V+\epsilon/2}\right) 
\quad\text{and}\quad 
\left(1-\frac{t}{x}\right)^{\rho_U+\epsilon/2} \!\!\! \leq \max\left(1,c^{\rho_U+\epsilon/2}\right).
$$
Hence we obtain, $U$ and $V$ being integrable, and since $\rho_V-\rho_U+\epsilon/2>0$,
$$
\limx\frac{\max\left(1,c^{\rho_V+\epsilon/2}\right)}{x^{\epsilon/2}}\int_0^{x/2} U(t)dt=0
\quad
\text{and}
\quad
\limx\frac{\max\left(1,c^{\rho_U+\epsilon/2}\right)}{x^{\rho_V-\rho_U+\epsilon/2}}\int_0^{x/2}V(s)ds=0,
$$
from which we deduce that, for any $\epsilon>0$, 
$ \displaystyle \limx\frac{U\ast V(x)}{x^{\rho+\epsilon}} =0$.

Applying Fatou's Lemma, then using that  $V\in \M$ with $\rho_V=\rho$, gives, for any $\epsilon$,
$$
\limx\frac{U\ast V(x)}{x^{\rho-\epsilon}} \ge \limx\int_0^1 \!\!\!\! U(t)\frac{V(x-t)}{x^{\rho-\epsilon}}dt  
 \geq  \liminfx\int_0^1\!\!\!\!  U(t)\frac{V(x-t)}{x^{\rho-\epsilon}}dt   \geq  \int_0^1 \!\!\!\! U(t)\limx\left(\frac{V(x-t)}{x^{\rho-\epsilon}}\right)dt =\infty.
$$
We can conclude that $U\ast V\in\M$ with $\rho_{U\ast V}=\rho_V$.

\item[$\triangleright$]
\emph{Assume $\rho_U<-1<0\leq\rho_V$}. Therefore $U$ is integrable on $\Rset^+$, but not $V$ (Properties \ref{propt:main:001}, (v)).
Choose $\rho=\rho_V$.

Using the change of variable $s=x-t$, we have, $\forall$ $x\geq 2x_0>x_0(>0)$,
\begin{eqnarray*}
\lefteqn{\frac{U\ast V(x)}{x^{\rho+\epsilon}} = \int_0^{x-x_0} \!\!\!\! U(t)\frac{V(x-t)}{x^{\rho+\epsilon}}dt+\int_{x-x_0}^x \!\!\!\!\!\!\! U(t)\frac{V(x-t)}{x^{\rho+\epsilon}}dt}  \\
 & & = \int_0^{x-x_0} \!\!\!\! U(t)\frac{V(x-t)}{x^{\rho+\epsilon}}dt+\int_0^{x_0}\!\!\!\!  V(s)\frac{U(x-s)}{x^{\rho+\epsilon}}ds \\
 & & \leq \int_0^{x-x_0} U(t)\frac{(x-t)^{\rho_V+\epsilon/2}}{x^{\rho+\epsilon}}dt+\int_0^{x_0} V(s)\frac{(x-s)^{\rho_U+\epsilon/2}}{x^{\rho+\epsilon}}ds \\
 & & = \frac{1}{x^{\epsilon/2}}\int_0^{x-x_0} U(t)\left(1-\frac{t}{x}\right)^{\rho_V+\epsilon/2} \!\!\!\! dt
 +\frac{1}{x^{\rho_V-\rho_U+\epsilon/2}}\int_0^{x_0}V(s)\left(1-\frac{s}{x}\right)^{\rho_U+\epsilon/2}\!\!\!\! ds\textrm{.}
\end{eqnarray*}
Noticing that for $0\leq t\leq x-x_0$, 
so $\displaystyle  \left(1-\frac{t}{x}\right)^{\rho_V+\epsilon/2}\leq1$,
and for $0\leq s\leq x_0<2x_0\leq x$, $\displaystyle 0< c< \frac{1}{2}\le 1-\frac{x_0}{x}\leq1-\frac{s}{x}\leq1$, 
so $\displaystyle  \left(1-\frac{s}{x}\right)^{\rho_U+\epsilon/2}\leq \max\left(1,c^{\,\rho_U+\epsilon/2}\right)$,
we obtain
$$
\frac{U\ast V(x)}{x^{\rho+\epsilon}} \leq\frac{1}{x^{\epsilon/2}}\int_0^{x-x_0} U(t)dt+\frac{\max\left(1,c^{\,\rho_U+\epsilon/2}\right)}{x^{\rho_V-\rho_U+\epsilon/2}}\int_0^{x_0}V(s)ds\, .
$$
Since $U$ is integrable, $V$ bounded on finite intervals, and $\rho_V-\rho_U+\epsilon/2>0$,  we have
$$
 \limx\frac{1}{x^{\epsilon/2}}\int_0^{x-x_0} U(t)dt=0 \quad \text{and}\quad
 \limx\frac{\max\left(1,c^{\,\rho_U+\epsilon/2}\right)}{x^{\rho_V-\rho_U+\epsilon/2}}\int_0^{x_0}V(t)dt=0.
 $$
therefore, for any $\epsilon>0$, we have $ \displaystyle \limx\frac{U\ast V(x)}{x^{\rho+\epsilon}} =0 $.

Applying Fatou's Lemma, then using that  $V\in \M$ with $\rho_V=\rho$, gives, for any $\epsilon$,
$$
\limx\frac{U\ast V(x)}{x^{\rho-\epsilon}} \ge \limx\int_0^1 \!\!\!\! U(t)\frac{V(x-t)}{x^{\rho-\epsilon}}dt  
 \geq  \liminfx\int_0^1\!\!\!\!  U(t)\frac{V(x-t)}{x^{\rho-\epsilon}}dt   \geq  \int_0^1 \!\!\!\! U(t)\limx\left(\frac{V(x-t)}{x^{\rho-\epsilon}}\right)dt =\infty.
$$
We can conclude that $U\ast V\in\M$ with $\rho_{U\ast V}=\rho_V$.

\item[$\triangleright$]
\emph{Assume $-1<\rho_U\leq\rho_V$}. Then both $U$ and $V$ are not integrable on $\Rset^+$ (Properties \ref{propt:main:001}, (v)).
Choose $\rho=\rho_U+\rho_V+1$.

Let $0<\epsilon<\rho_U+1$.
Since $V$ is not integrable on $\Rset^+$,
we have $\displaystyle \int_0^xV(t)dt  \underset{x\to\infty}{\rightarrow} \infty$.
So we can apply \thelr\ and obtain
$$
\limx\frac{\int_0^xV(t)dt}{x^{\rho_V+1+\epsilon}}=
\limx\frac{\left(\int_0^xV(t)dt\right)'}{\left(x^{\rho_V+1+\epsilon}\right)'}=
\limx\frac{V(x)}{(\rho_V+1+\epsilon)x^{\rho_V+\epsilon}}=0
$$
and
$$
\limx\frac{\int_0^xV(t)dt}{x^{\rho_V+1-\epsilon}}=
\limx\frac{\left(\int_0^xV(t)dt\right)'}{\left(x^{\rho_V+1-\epsilon}\right)'}=
\limx\frac{V(x)}{(\rho_V+1-\epsilon)x^{\rho_V-\epsilon}}=\infty\textrm{,}
$$
from which we deduce that $\displaystyle W_V(x):=\int_0^xV(t)dt\in\M$ with $\M$-index $\rho_V+1$.

We obtain in the same way that $\displaystyle W_U(x):=\int_0^xU(t)dt\in\M$ with $\M$-index $\rho_U+1$.

We have, via the change of variable $s=x-t$, $\forall\ x\geq2 x_0>0$,
\begin{eqnarray*}
\lefteqn{\frac{U\ast V(x)}{x^{\rho+\epsilon}} = \int_0^{x/2} U(t)\frac{V(x-t)}{x^{\rho+\epsilon}}dt+\int_{x/2}^x U(t)\frac{V(x-t)}{x^{\rho+\epsilon}}dt} \nonumber \\
 & & \le \frac{1}{x^{\rho_U+1+\epsilon/2}}\int_0^{x/2} U(t)\left(1-\frac{t}{x}\right)^{\rho_V+\epsilon/2}dt
+\frac{1}{x^{\rho_V+1+\epsilon/2}}\int_0^{x/2} V(s)\left(1-\frac{s}{x}\right)^{\rho_U+\epsilon/2}ds \nonumber \\
& & \le \max\left(1,c^{\rho_V+\epsilon/2}\right)\frac{W_U(x/2)}{x^{\rho_U+1+\epsilon/2}}+\max\left(1,c^{\rho_U+\epsilon/2}\right)\frac{W_V(x/2)}{x^{\rho_V+1+\epsilon/2}}
\end{eqnarray*}
and
\begin{eqnarray*}
\lefteqn{\frac{U\ast V(x)}{x^{\rho-\epsilon}} = \int_0^{x/2} U(t)\frac{V(x-t)}{x^{\rho-\epsilon}}dt+\int_{x/2}^x U(t)\frac{V(x-t)}{x^{\rho-\epsilon}}dt} \nonumber \\
 & & \ge \frac{1}{x^{\rho_U+1-\epsilon/2}}\int_0^{x/2} U(t)\left(1-\frac{t}{x}\right)^{\rho_V-\epsilon/2}dt
+\frac{1}{x^{\rho_V+1-\epsilon/2}}\int_0^{x/2} V(s)\left(1-\frac{s}{x}\right)^{\rho_U-\epsilon/2}ds \nonumber \\
& & \ge \min\left(1,c^{\rho_V-\epsilon/2}\right)\frac{W_U(x/2)}{x^{\rho_U+1-\epsilon/2}}+\min\left(1,c^{\rho_U-\epsilon/2}\right)\frac{W_V(x/2)}{x^{\rho_V+1-\epsilon/2}}
\end{eqnarray*}
since, for $0\leq t \leq x/2$, {\it i.e.} $\displaystyle 0<c<\frac{1}{2}\leq1-\frac{t}{x}\leq1$, 
$$
\min\left(1,c^{\rho_V-\epsilon/2}\right)\leq\left(1-\frac{t}{x}\right)^{\rho_V-\epsilon/2}\!\!\! \leq\left(1-\frac{t}{x}\right)^{\rho_V+\epsilon/2}\!\!\! \leq \max\left(1,c^{\rho_V+\epsilon/2}\right) 
$$
and
$$
\min\left(1,c^{\rho_U-\epsilon/2}\right)\leq\left(1-\frac{t}{x}\right)^{\rho_U-\epsilon/2} \!\!\!\leq\left(1-\frac{t}{x}\right)^{\rho_U+\epsilon/2} \!\!\! \leq \max\left(1,c^{\rho_U+\epsilon/2}\right).
$$
Hence, for any $0<\epsilon<\rho_U+1$, we have $ \displaystyle \limx\frac{U\ast V(x)}{x^{\rho+\epsilon}}=0$ and $ \displaystyle \limx\frac{U\ast V(x)}{x^{\rho-\epsilon}}=\infty$.
We can conclude that $U\ast V\in\M$ with $\rho_{U\ast V}=\rho_U+\rho_V+1$.
\end{itemize}

\item {\it Proof of (iii)}~

It is straightforward, since we can write,  with $y=V(x)\to\infty$ as $x\to\infty$, 
$$
\limx\frac{\log(U(V(x)))}{\log(x)}=
\limy\frac{\log(U(y))}{\log(y)}\times
\limx\frac{\log(V(x))}{\log(x)}=\rho_U\,\rho_V
$$
Hence we obtain $\rho_{U\circ V}=\rho_U\,\rho_V$.
\end{itemize}
\end{proof}


\subsection{Proofs of results concerning $\mathcal{M}_\infty$ and $\mathcal{M}_{-\infty}$}

\begin{proof}[Proof of Theorem\,\ref{teo:main:001extension}]~

%
It is enough to prove \eqref{eq:main:000bextension01}
because by this equivalence and Properties \ref{propt:main:001extension}, (i), one has
$$
U\in\M_{-\infty}\quad\Longleftrightarrow\quad 1/U\in\M_{\infty}\quad\Longleftrightarrow\quad\lim_{x\rightarrow\infty}-\frac{\log\left(1/U(x)\right)}{\log(x)}=\infty
\quad\Longleftrightarrow\quad\lim_{x\rightarrow\infty}-\frac{\log\left(U(x)\right)}{\log(x)}=-\infty\textrm{,}
$$
i.e. \eqref{eq:main:000bextension02}.
\begin{itemize}
\item
Let us prove that
$
\displaystyle U\in\M_\infty\quad\Longrightarrow\quad\lim_{x\rightarrow\infty}\frac{\log\left(U(x)\right)}{\log(x)}=-\infty
$.

Suppose $U\in\M_\infty$.
This implies that for all $\rho\in\Rset$, one has $\displaystyle \limx\frac{U(x)}{x^\rho}=0$, i.e.
for all $\epsilon>0$ there exists $x_0>1$ such that, for $x\geq x_0$,
$
U(x)\leq\epsilon x^\rho
$
which implies
$
\displaystyle
\frac{\log\left(U(x)\right)}{\log(x)}\leq\frac{\log(\epsilon)}{\log(x)}+\rho
$, 
hence
$
\displaystyle \limx\frac{\log\left(U(x)\right)}{\log(x)}\leq\rho
$
and the statement follows since the argument applies for all $\rho\in\Rset$.
\item
Now let us prove that
%
$
\displaystyle \lim_{x\rightarrow\infty}-\frac{\log\left(U(x)\right)}{\log(x)}=\infty\quad\Longrightarrow\quad U\in\M_\infty
$.

For any $\rho\in \Rset$, we can write
$$
\lim_{x\rightarrow\infty}-\frac{\log\left(\frac{U(x)}{x^{\rho}}\right)}{\log(x)}=
\lim_{x\rightarrow\infty}\left(-\frac{\log\left(U(x)\right)}{\log(x)}+\rho\right)=\infty\quad\textrm{under the hypothesis,}
$$
which implies that $U(x)\big/x^{\rho}<1$ and hence $\displaystyle  \lim_{x\rightarrow\infty}\frac{U(x)}{x^{\rho}}=0$.
\end{itemize}
\end{proof}

\begin{proof}[Proof of Theorem \ref{teo:main:002extension}]~

\begin{itemize}
\item\emph{Proof of (i)-(a)}

Suppose $U\in\M_\infty$.
Then, by definition (\ref{M+}), for any $\rho\in\Rset$, 
$
\displaystyle \lim_{x\rightarrow\infty}x^\rho U(x)=0\textrm{}
$, 
which implies that for $c>0$, 
there exists $x_0>1$ such that, for all $x\geq x_0$,
$
\displaystyle U(x)\leq cx^{-\rho}\textrm{,}
$
%
from which we deduce that
$$
\int_{x_0}^{\infty}x^{r-1}U(x)dx\leq c\int_{x_0}^{\infty}x^{r-1-\rho}dx
$$
which is finite whenever $r<\rho$.
This result holds also on $(1;\infty)$ since $U$ is bounded on finite intervals.

Thus we conclude that $\kappa_U=\infty$, $\rho$ being any real number.
\item \emph{Proof of (i)-(b)}

Note that $U$ is integrable on $\Rset^+$ since $\displaystyle \int_1^\infty x^{r-1}U(x)dx<\infty$, for any $r\in\Rset$, in particular for $r=1$. Moreover $U$ is bounded on finite intervals.

For $r>0$, we have, via the continuity of $U$,
$$
\int_0^\infty x^{r+1}dU(x) = (r+1)\int_0^\infty\int_0^xy^{r}dy\, dU(x)=(r+1)\int_0^\infty y^r\left(\int_y^{\infty} dU(x)\right)\, dy
$$
which implies, since $\displaystyle \limx U(x)=0$, that 
\begin{equation}\label{eq:20141106:001}
-\int_0^\infty x^{r+1}dU(x) = (r+1)\int_0^\infty y^{r}U(y)dy,
\end{equation}
which is positive and finite.

Now, for $t>0$, we have, integrating by parts and using again the continuity of $U$,
$$
t^{r+1}U(t)=(r+1)\int_0^tx^{r}U(x)dx+\int_0^tx^{r+1}dU(x)
$$
where the integrals on the right hand side of the equality are finite as $t\to\infty$ and their sum tends to 0 via \eqref{eq:20141106:001}.
This implies that, $\forall r>0$, $t^{r+1}U(t)\to0$ as $t\to\infty$.

For $r\leq0$, we have, for $t\geq1$, using the previous result, $t^{r+1}U(t)\leq t^2U(t)\to0$ as $t\to\infty$.

This completes the proof that $U\in\M_\infty$.
\item \emph{Proof of (ii)-(a)}

Suppose $U\in\M_{-\infty}$.
%
%
Then, by definition (\ref{M-}), for any $\rho\in\Rset$, we have
$
\displaystyle \lim_{x\rightarrow\infty}\frac{U(x)}{x^\rho}=\infty\textrm{}
$, 
which implies that for $c>0$, 
there exists $x_0>1$ such that, for all $x\geq x_0$,
$
\displaystyle U(x)\geq cx^{\rho}\textrm{,}
$
%
from which we deduce that, $U$ being bounded on finite intervals,
$$
\int_{1}^{\infty}x^{r-1}U(x)dx\geq c\int_{x_0}^{\infty}x^{r-1+\rho}dx
$$
which is infinite whenever $r\geq-\rho$.
%

The argument applying for any $\rho$, we conclude that $\kappa_U=-\infty$.
\item\emph{Proof of (ii)-(b)}

Let $r\geq0$. We can write, for $s+2<0$ and $t>1$,
\begin{eqnarray*}
0 & \geq & -\int_{1}^t x^{s+1}d\left(x^rU(x)\right)\qquad (x^rU(x) \, \textrm{being non-decreasing}) \\
 & = & \int_{1}^t\left(\int_x^t d\left(y^{s+1}\right)\, - t^{s+1}\right)d\left(x^rU(x)\right) \\
 & = & \int_1^t y^{s+1}\left( \int_{1}^y d\left(x^rU(x)\right)\right) dy \,-t^{s+1}  \int_{1}^t d\left(x^rU(x)\right)  \\
 & = & \int_1^t y^{s+r+1}U(y)dy\, - \frac{t^{s+2}-1}{s+2} \,U(1)\, - t^{s+1}\left(t^rU(t)-U(1)\right)\quad \textrm{( $U$ being continue)}.
\end{eqnarray*}
Hence we obtain, as $t\to\infty$, $\displaystyle t^{s+r+1}U(t)\to\infty$ since $\displaystyle \int_1^t y^{s+r+1}U(y)dy\to\infty$ and $\displaystyle \frac{t^{s+2}}{s+2}+t^{s+1}\to0$ (under the assumption $s<-2$).\\
This implies that $U\in\M_{-\infty}$ since $s+r+1\in\Rset$.
\end{itemize}
\end{proof}

\begin{proof}[Proof of Remark \ref{rmq:20141107:001} -1]~

Set $\displaystyle A=\int_1^{\infty}e^{-x}dx=e^{-1}$ and let us prove that $U\in\M_\infty$.\\
If $r>0$, then 
\begin{eqnarray*}
\lefteqn{\int_1^\infty x^{r}U(x)dx\leq A+\sum_{n=1}^\infty\int_{n}^{n+1/n^n} x^{r}U(x)dx=A+\sum_{n=1}^\infty\int_{n}^{n+1/n^n} x^{r-1}dx} \\
 & & \leq A+\sum_{n=1}^\infty\int_{n}^{n+1/n^n} x^{\left\lceil r\right\rceil-1}dx
=A+\frac{1}{\left\lceil r\right\rceil}\sum_{n=1}^\infty\left((n+1/n^n)^{\left\lceil r\right\rceil}-n^{\left\lceil r\right\rceil}\right)dx \\
 & & =A+\frac{1}{\left\lceil r\right\rceil}\sum_{n=1}^\infty n^{-(n-1)\left\lceil r\right\rceil-1} \sum_{k=0}^{\left\lceil r\right\rceil-1}(1+1/n^{n-1})^{k}  <\infty \, .
\end{eqnarray*}
If $r\leq0$, then we can write $\displaystyle \int_1^\infty x^{r}U(x)dx \, \leq \, \int_1^\infty xU(x)dx$, which is finite using the previous result with $r=1$.

Now, let us prove $U\not\in\M_{\infty}$ by contradiction.\\
Suppose $U\in\M_{\infty}$. Then Theorem \ref{teo:main:001extension} implies that
$\displaystyle \limx\frac{\log\left(U(x)\right)}{\log(x)}=-\infty$, which contradicts
$$
\limn\frac{\log\left(U(n)\right)}{\log(n)}=
\limn\frac{\log\left(1/n\right)}{\log(n)}=-1>-\infty\textrm{.}
$$
\end{proof}


\begin{proof}[Proof of Theorem \ref{teo:main:003extension}]~
\begin{itemize}
\item \emph{Proof of (i)}

Suppose $U\in\mathcal{M}_\infty$.
By Theorem \ref{teo:main:001extension}, we have
%
%
$
\displaystyle \lim_{x\rightarrow\infty}-\frac{\log(U(x))}{\log(x)}=\infty\textrm{.}
$
It implies that there exists $b>1$ such that, for $x\geq b$,
$
\displaystyle \beta(x):=-\frac{\log(U(x))}{\log(x)}>0\textrm{.}
$
Defining, for $x\geq b$, 
$
\alpha(x):=\beta(x)\,\log(x)\textrm{,}
$
gives (i).
\item \emph{Proof of (ii)}

Suppose $U\in\mathcal{M}_{-\infty}$.
By Properties \ref{propt:main:001extension}, (i), $1/U\in\mathcal{M}_{\infty}$.
Applying the previous result to $1/U$ implies that there exists a positive function $\alpha$ satisfying $\alpha(x)/\log(x)\underset{x\rightarrow\infty}{\rightarrow}\infty$ such that $1/U(x)=\exp(-\alpha(x))$, $x\geq b$ for some $b>1$.
Hence we get $U(x)=\exp(-\alpha(x))$, $x\geq b$, as required.

\item \emph{Proof of (iii)}


Assume that $U$ satisfies, for $x\geq b$,
$
U(x)=\exp(-\alpha(x))
$,
for some $b>1$ and $\alpha$ satisfying $\alpha(x)/\log(x)\underset{x\rightarrow\infty}{\rightarrow}\infty$.
A straightforward computation gives
$
\displaystyle
\lim_{x\rightarrow\infty}-\frac{\log(U(x))}{\log(x)}=
\lim_{x\rightarrow\infty}\frac{\alpha(x)}{\log(x)}=\infty\textrm{.}
$
Hence $U\in\M_\infty$.

We can proceed exactly in the same way when supposing
that $U$ satisfies, for $x\geq b$,
$
U(x)=\exp(\alpha(x))
$
for some $b>1$ and $\alpha$ satisfying $\alpha(x)/\log(x)\underset{x\rightarrow\infty}{\rightarrow}\infty$,
to conclude that
$U\in\M_{-\infty}$.
\end{itemize}
\end{proof}

\begin{proof}[Proof of Properties \ref{propt:main:001extension}]~
\begin{itemize}
\item
\emph{Proof of (i)}

It is straightforward since, for $\rho\in\Rset$, 
$
\displaystyle \lim_{x\rightarrow\infty}\frac{U(x)}{x^{\rho}}=0\textrm{}
\quad\Longleftrightarrow\quad
\lim_{x\rightarrow\infty}\frac{1/U(x)}{x^{-\rho}}=\infty\textrm{.}
$
\item \emph{Proof of (ii)}
\begin{itemize}
\item[$\triangleright$]
Suppose $(U,V)\in\M_{-\infty}\times\M_\fty$ with $\rho_V$ defined in (\ref{Mkappa}).

Let $\epsilon>0$.
Writing
$
\displaystyle
\frac{V(x)}{U(x)}
=\frac{V(x)}{x^{\rho_V+\epsilon}}\,\left(\frac{U(x)}{x^{\rho_V+\epsilon}}\right)^{-1}\textrm{,}
$
we obtain
$
\displaystyle
\lim_{x\rightarrow\infty}\frac{V(x)}{U(x)} 
=0\textrm{}
$
since $V\in\M$ with $\rho_V$ satisfying (\ref{Mkappa}) and $U$ satisfies (\ref{M-}) with $\rho_U=\rho_V+\epsilon\in\Rset$.
\item[$\triangleright$]
Suppose $(U,V)\in\M_{-\infty}\times\M_\infty$.

Let $\rho>0$. We have
$
\displaystyle
\lim_{x\rightarrow\infty}\frac{V(x)}{U(x)} = \lim_{x\rightarrow\infty}\frac{V(x)}{x^{\rho}}\,\left(\frac{U(x)}{x^{\rho}}\right)^{-1}
=0\textrm{}
$
since $V$ satisfies (\ref{M+}) and $U$ 
(\ref{M-}).
\item[$\triangleright$] Suppose $(U,V)\in\M_\fty\times\M_\infty$ with $\rho_U$ defined in (\ref{Mkappa}).

By Properties \ref{propt:main:001}, (iv), and Properties \ref{propt:main:001extension}, (i), we have $(1/U,1/V)\in\M_\fty\times\M_{-\infty}$.
The result follows because
$
\displaystyle
\limx\frac{V(x)}{U(x)}=
\limx\frac{1/U(x)}{1/V(x)}=0
$.
\end{itemize}
\item The proof of (iii) is immediate.
\end{itemize}
\end{proof}

\begin{proof}[Proof of Properties \ref{propt:20140630ext}] ~
Let $U$, $V$ $\in$ $\M$ with $\M$-index $\kappa_U$ and $\kappa_V$ respectively.
\begin{itemize}
\item {\it Proof of (i)}~

It is straightforward as
%
$
\displaystyle
\limx\frac{\log\left(U(x)\,V(x)\right)}{\log(x)}=
\limx\left(\frac{\log\left(U(x)\right)}{\log(x)}+\frac{\log\left(V(x)\right)}{\log(x)}\right)
$.
\item {\it Proof of (ii)}~

We distinguish the next three cases.\\
(a) {\it Let $U\in\M_{\infty}$ and $V\in$ $\M$ with $\rho_V\notin [-1, 0)$.}~

Let $W(x)=x^\eta 1_{(x\ge 1)}+ 1_{(0<x<1)}$,  with $\eta=-2$ if $\rho_V\geq0$, or $\eta=\rho_V-1$ if $\rho_V<-1$.
Note that $W\in\M$ with $\rho_W=\eta < \rho_V$.

By Properties \ref{propt:main:001extension}, (ii), $\displaystyle \limx\frac{U(x)}{W(x)}=0$,
so for $0<\delta<1$, there exists $x_0\geq1$ such that, for all $x\geq x_0$, $U(x)\leq  \delta W(x)$.

Consider $Z$ defined by $Z(x)=U(x)1_{(0<x<x_0)}+ W(x)1_{(x\ge x_0)}$, which satisfies $Z\geq U$ and $Z\in\M$ with $\rho_Z=\rho_W=\eta<\rho_V$.  
Applying Properties \ref{propt:20140630}, (ii), gives $Z\ast  V\in\M$ with $\rho_{Z\ast  V}=\rho_Z \vee \rho_V=\rho_V$ (note that the restriction on $\rho_v$ corresponds to the condition given in Properties \ref{propt:20140630}, (ii)).

We deduce that, for any $x>0$, $U\ast  V(x)\le Z\ast  V(x)$, and, for  $\epsilon>0$,
$$
\frac{U\ast  V(x)}{x^{\rho_V+\epsilon}}\leq
\frac{Z\ast  V(x)}{x^{\rho_V+\epsilon}}\ \underset{x\to\infty}{\rightarrow}\ \ 0\,.
$$
Moreover, applying Fatou's Lemma gives
$$
\limx\frac{U\ast  V(x)}{x^{\rho_V-\epsilon}} \ge \limx\int_0^1 \!\!\!\! U(t)\frac{V(x-t)}{x^{\rho_V-\epsilon}}dt \ge \liminfx\int_0^1  \!\!\!\! U(t)\frac{V(x-t)}{x^{\rho_V-\epsilon}}dt
\geq\int_0^1  \!\!\!\! U(t)\limx\left(\frac{V(x-t)}{x^{\rho_V-\epsilon}}\right)dt=\infty\textrm{.}
$$
Therefore, $U\ast V\in \M$ with $\M$-index $\rho_{U\ast V}=\rho_V$.\\

(b) {\it $(U,V)\in\M_\infty\times\M_\infty$, then $U\ast V\in\M_\infty$}

Let $\rho\in\Rset$.

Consider $U\in\M_\infty$.
We have, applying Theorem \ref{teo:main:001extension},
$
\displaystyle \limx\frac{\log(U(x))}{\log(x)}=-\infty
$.
Rewriting this limit as
$$
\displaystyle \limx\frac{\log(U(x))}{\log(1/x)}=\infty
$$
we deduce that, for $c\geq|\rho|+1>0$, there exists $x_U>1$ such that, for $x\geq x_U$, $\log(U(x))\leq c\log(1/x)$, i.e. $U(x)\leq x^{-c}$.
On $V\in\M_\infty$, a similar reasoning leads to that there exists $x_V>1$ such that, for $x\geq x_V$, $V(x)\leq x^{-c}$.

Using the change of variable $s=x-t$, we have, $\forall\ x\geq2\max(x_U,x_V)>0$,

\begin{eqnarray*}
\lefteqn{\frac{U\ast V(x)}{x^{\rho}} = \int_0^{x/2} U(t)\frac{V(x-t)}{x^{\rho}}dt+\int_{x/2}^x U(t)\frac{V(x-t)}{x^{\rho}}dt} \\
 & & \le \frac{1}{x^{\rho+c}}\int_0^{x/2} U(t)\Big(1-\frac{t}{x}\Big)^{-c}dt
+\frac{1}{x^{\rho+c}}\int_0^{x/2} V(s)\Big(1-\frac{s}{x}\Big)^{-c}ds \\
& & \le \frac{2^c}{x^{\rho+c}}\int_0^{x/2} U(t)dt+\frac{2^c}{x^{\rho+c}}\int_0^{x/2} V(s)ds
\end{eqnarray*}
since, for $0\leq t \leq x/2$, {\it i.e.} $\displaystyle 0<\frac{1}{2}\leq1-\frac{t}{x}\leq1$, 
$
\displaystyle \Big(1-\frac{t}{x}\Big)^{-c}\!\!\! \leq 2^c
$.

This implies, via the integrability of $U$ and $V$, for $\rho\in\Rset$, $\displaystyle \limx\frac{U\ast V(x)}{x^{\rho}}=0$.
Hence $U\ast V\in\M_\infty$.

(c) {\it Let $U\in\M_{-\infty}$ and $V\in$ $\M$ or $\M_{\pm\infty}$.}~

We apply Fatou's Lemma, as in (a), to obtain, for any $\rho\in\Rset$,
$$
\limx\frac{U\ast V(x)}{x^{\rho}}  \ge 
\liminfx\int_0^1  V(t)\frac{U(x-t)}{x^{\rho}}dt
 \geq \int_0^1 V(t)\limx\left(\frac{U(x-t)}{x^{\rho}}\right)dt=\infty\textrm{.}
$$
We conclude that $U\ast V\in\M_{-\infty}$.

\item {\it Proof of (iii)}~

First, note that if $V$ $\in$ $\M_{-\infty}$, then $\displaystyle \limx V(x)=\infty$.
Hence writing
$$
\frac{\log\left(U(V(x))\right)}{\log(x)}=
\frac{\log\left(U(y)\right)}{\log(y)}\,
\times
\frac{\log\left(V(x)\right)}{\log(x)}\textrm{,\quad with $y=V(x)$}
$$
allows to conclude.
\end{itemize}
\vspace{-2ex}
\end{proof}

\subsection{Proofs of results concerning $\Oset$}  \label{ProofsofresultsconcerningOset}

Let us recall the \PBdH, that we will need to prove Theorem \ref{teo:20140801:001}.

Let us define the Generalized Pareto Distribution (GPD)
$$
G_\xi(x)=\left\{
\begin{array}{ll}
(1+\xi x)^{-1/\xi} & \xi\in\Rset\textrm{, }\xi\neq0\textrm{, }1+\xi x>0 \\
 & \\
e^{-x} & \xi=0\textrm{, }x\in\Rset\textrm{.}
\end{array}
\right.
$$
\begin{teo}\label{tpbdh}\PBdH\ (see e.g. Theorem 3.4.5 in \cite{embrechts1997}, \PBdH)~

For $\xi\in\Rset$, the following assertions are equivalent:
\begin{itemize}
\item[(i)]
$F\in DA(\exp(-G_\xi))$
\item[(ii)]
There exists a positive function $a>0$ such that for $1+\xi x>0$,
$$
\lim_{u\rightarrow\infty}\frac{\barF(u+x\,a(u))}{\barF(u)}=G_\xi(x)\textrm{.}
$$
\end{itemize}
\end{teo}

Note that Theorem \ref{teo:20140801:001} refers to distributions $F$ with endpoint $x^* =\sup\{x:F(x)<1\}=\infty$.

\begin{proof}[Proof of Theorem \ref{teo:20140801:001}:]~

Let us prove this theorem by contradiction,
assuming that $F$ satisfies the \PBdH\ and that $\barF$ satisfies $\mu(\barF)<\nu(\barF)$.
Note that $x^* =\infty$.
We consider the two possibilities given in (i) and (ii) in Theorem \ref{tpbdh}.
\begin{itemize}
\item \emph{Assume that $F$ satisfies Theorem \ref{tpbdh}, (i), with $\xi\geq0$ (because $x^* =\infty$).}

Let $\epsilon>0$.
By Theorem \ref{tpbdh}, (ii), there exists $u_0>0$ such that, for $u\geq u_0$ and $x\geq0$,
\begin{equation}\label{eq:20140803:001}
\frac{\barF(u+x)}{\barF(u)\,G_\xi(x/a(u))}\leq1+\epsilon\textrm{.}
\end{equation}
By the definition of upper order, we have that there exists a sequence $(x_n)_{n\in\Nset}$ satisfying $x_n\rightarrow\infty$ as $n\to\infty$ such that
\begin{eqnarray*}
\lefteqn{\nu(\barF)
=\lim_{x_n\rightarrow\infty}\frac{\log\left(\barF(u+x_n)\right)}{\log(u+x_n)}
=\lim_{x_n\rightarrow\infty}\frac{\log\left(\barF(u+x_n)\right)}{\log(x_n)}} \\
 & & \leq\lim_{x_n\rightarrow\infty}\frac{\log\left((1+\epsilon)\,\barF(u)\,G_\xi(x_n/a(u))\right)}{\log(x_n)}\quad\textrm{by \eqref{eq:20140803:001}} \\
 & & =\lim_{x_n\rightarrow\infty}\frac{\log\left(\barF(u)\right)}{\log(x_n)}+\lim_{x_n\rightarrow\infty}\frac{\log\left(G_\xi(x_n/a(u))\right)}{\log(x_n)} \\
 & & =\left\{
\begin{array}{ll}
-\frac{1}{\xi}\lim_{x_n\rightarrow\infty}\frac{\log\left(1+\xi\,x_n/a(u)\right)}{\log(x_n)} & \textrm{if $\xi>0$} \\
-\lim_{x_n\rightarrow\infty}\frac{x_n/a(u)}{\log(x_n)} & \textrm{if $\xi=0$}
\end{array}
\right. \\
 & & =\left\{
\begin{array}{ll}
-\frac{1}{\xi} & \textrm{if $\xi>0$} \\
-\infty & \textrm{if $\xi=0$.}
\end{array}
\right.\textrm{}
\end{eqnarray*}
If $\xi>0$, we conclude that $\nu(\barF)\leq-1/\xi$. A similar procedure provides $\mu(\barF)\geq-1/\xi$.
Hence we conclude $\mu(\barF)=\nu(\barF)$ which contradicts $\mu(\barF)<\nu(\barF)$.

If $\xi=0$, we conclude that $-\infty\leq\mu(\barF)\leq\nu(\barF)\leq-\infty$.
Hence we conclude $\mu(\barF)=\nu(\barF)=-\infty$ which contradicts $\mu(\barF)<\nu(\barF)$.
\item \emph{Assuming that $F$ satisfies Theorem \ref{tpbdh}, (ii),}
and following the previous proof (done when assuming (i)), we deduce that $\mu(\barF)=\nu(\barF)$ which contradicts $\mu(\barF)<\nu(\barF)$.
\end{itemize}
\end{proof}

\begin{proof}[Proof of Example \ref{exm:20140802:001}]~

Let $x\in[x_n,x_{n+1})$, $n\geq1$.
We can write
\begin{equation}\label{eq:20140610:002}
\frac{\log\left(U(x)\right)}{\log(x)}=\frac{\log\left(x_n^{\alpha(1+\beta)}\right)}{\log(x)}=\alpha(1+\beta)\frac{\log\left(x_n\right)}{\log(x)}\textrm{.}
\end{equation}
Since
$
\log(x_n)\leq\log(x)<\log(x_{n+1})=(1+\alpha)\log(x_n)
$, 
we obtain
$$
\frac{\alpha(1+\beta)}{1+\alpha}  < \frac{\log\left(U(x)\right)}{\log(x)} \leq \alpha(1+\beta)\textrm{,} \quad\textrm{if\quad $1+\beta>0$}
$$
and
$$
\alpha(1+\beta) \leq \frac{\log\left(U(x)\right)}{\log(x)} < \frac{\alpha(1+\beta)}{1+\alpha}\textrm{,} \quad\textrm{if\quad $1+\beta<0$}
$$
from which we deduce
$$
\mu(U)\geq\frac{\alpha(1+\beta)}{1+\alpha}\quad\textrm{and}\quad\nu(U)\leq\alpha(1+\beta)\textrm{,} \quad\textrm{if\quad $1+\beta>0$}
$$
and
$$
\mu(U)\geq\alpha(1+\beta)\quad\textrm{and}\quad\nu(U)\leq\frac{\alpha(1+\beta)}{1+\alpha}\textrm{,} \quad\textrm{if\quad $1+\beta<0$.}
$$
Moreover,  taking $x=x_n$ in \eqref{eq:20140610:002} leads to
$$
\lim_{n\rightarrow\infty}\frac{\log\left(U(x_n)\right)}{\log(x_n)}=\alpha(1+\beta)
$$
which implies
$$
\nu(U)\geq\alpha(1+\beta)\textrm{,} \quad\textrm{if\quad $1+\beta>0$}
$$
and
$$
\mu(U)\leq\alpha(1+\beta)\textrm{,} \quad\textrm{if\quad $1+\beta<0$.}
$$
Hence, to conclude, it remains to prove that
$$
\mu(U)\leq\frac{\alpha(1+\beta)}{1+\alpha}\textrm{,} \quad\textrm{if\quad $1+\beta>0$,}
\quad\textrm{and}\quad
\nu(U)\geq\frac{\alpha(1+\beta)}{1+\alpha}\textrm{,} \quad\textrm{if\quad $1+\beta<0$.}
$$
If $1+\beta>0$, the function $\log\left(U(x)\right)/\log(x)$ is strictly decreasing continuous on $(x_n;x_{n+1})$ reaching the supremum value $\alpha(1+\beta)$ 
and the infimum value $\alpha(1+\beta)/(1+\alpha)$. Hence, for $\delta>0$ such that 
$$
\frac{\alpha(1+\beta)}{1+\alpha}<\frac{\alpha(1+\beta)}{1+\alpha}+\delta<\alpha(1+\beta)\textrm{,}
$$
there exists $x_n<y_n<x_{n+1}$ satisfying
$$
\frac{\log\left(U(y_n)\right)}{\log(y_n)}=\frac{\alpha(1+\beta)}{1+\alpha}+\delta\textrm{.}
$$
Since $y_n\rightarrow\infty$ as $n\rightarrow\infty$ because $x_n\rightarrow\infty$ as $n\rightarrow\infty$, 
$
\displaystyle
\mu(U)\leq\lim_{n\rightarrow\infty}\frac{\log\left(U(y_n)\right)}{\log(y_n)}=\frac{\alpha(1+\beta)}{1+\alpha}+\delta
$
follows.
Hence we conclude $\displaystyle \mu(U)\leq\frac{\alpha(1+\beta)}{1+\alpha}$ since $\delta$ is arbitrary.

If $1+\beta<0$, a similar development to the case $1+\beta>0$ allows proving $\displaystyle \nu(U)\geq\frac{\alpha(1+\beta)}{1+\alpha}$.

Moreover, if $1+\beta<0$ we have that $U$ is a tail of distribution.
Let us check that the rv having a tail of distribution $\barF=U$ has a finite $s$th moment whenver $0\leq s<-\alpha(1+\beta)/(1+\alpha)$.

Let $s\geq 0$. We have
\begin{eqnarray*}
\lefteqn{\int_0^\infty x^sdF(x) = \sum_{n=1}^\infty x_n^s\left(U(x_{n}^-)-U(x_n^+)\right)} \\
 & &  = \sum_{n=2}^\infty x_n^s\left(x_{n-1}^{\alpha(1+\beta)}-x_{n}^{\alpha(1+\beta)}\right)
=\sum_{n=2}^\infty x_n^s\left(x_{n}^{\frac{\alpha(1+\beta)}{1+\alpha}}-x_{n}^{\alpha(1+\beta)}\right)
\leq \sum_{n=2}^\infty x_{n}^{s+\frac{\alpha(1+\beta)}{1+\alpha}}<\infty
\end{eqnarray*}
because $s<-\alpha(1+\beta)/(1+\alpha)$.


Note that if $s\geq-\alpha(1+\beta)/(1+\alpha)$, $\displaystyle \int_0^\infty x^sdF(x)=\infty$.
\end{proof}

\begin{proof}[Proof of Example \ref{exm:20140802:002}]~

If $\alpha>0$, $\nu(U)=\infty$ comes from
$$
\nu(U)=\limsupx\frac{\log\left(U(x)\right)}{\log(x)}\geq\lim_{x_n\rightarrow\infty}\frac{\log\left(U(x_n)\right)}{\log(x_n)}
=\lim_{x_n\rightarrow\infty}\frac{\alpha x_n\,\log(2)}{\log(x_n)}=\infty\textrm{,}
$$
and, if $\alpha<0$, $\mu(U)=-\infty$ comes from
$$
\mu(U)=\liminfx\frac{\log\left(U(x)\right)}{\log(x)}\leq\lim_{x_n\rightarrow\infty}\frac{\log\left(U(x_n)\right)}{\log(x_n)}
=\lim_{x_n\rightarrow\infty}\frac{\alpha x_n\,\log(2)}{\log(x_n)}=-\infty\textrm{.}
$$
Next, let $\epsilon>0$ be small enough. Then, we have, if $\alpha>0$,
\begin{eqnarray*}
\lefteqn{\mu(U)=\liminfx\frac{\log\left(U(x)\right)}{\log(x)}\leq\lim_{x_n\rightarrow\infty}\frac{\log\left(U(x_n-\epsilon)\right)}{\log(x_n-\epsilon)}
} \\
 & & =\lim_{x_n\rightarrow\infty}\frac{\log\left(2^{\alpha x_{n-1}}\right)}{\log(2^{x_{n-1}/c})}\,\frac{\log(2^{x_{n-1}/c})}{\log(2^{x_{n-1}/c}-\epsilon)}
=\lim_{x_n\rightarrow\infty}\frac{\log\left(2^{\alpha x_{n-1}}\right)}{\log(2^{x_{n-1}/c})}=\alpha c\textrm{,}
\end{eqnarray*}
and, if $\alpha<0$,
\begin{eqnarray*}
\lefteqn{\nu(U)=\limsupx\frac{\log\left(U(x)\right)}{\log(x)}\geq\lim_{x_n\rightarrow\infty}\frac{\log\left(U(x_n-\epsilon)\right)}{\log(x_n-\epsilon)}
} \\
 & & =\lim_{x_n\rightarrow\infty}\frac{\log\left(2^{\alpha x_{n-1}}\right)}{\log(2^{x_{n-1}/c})}\,\frac{\log(2^{x_{n-1}/c})}{\log(2^{x_{n-1}/c}-\epsilon)}
=\lim_{x_n\rightarrow\infty}\frac{\log\left(2^{\alpha x_{n-1}}\right)}{\log(2^{x_{n-1}/c})}=\alpha c\textrm{.}
\end{eqnarray*}
It remains to prove that, if $\alpha>0$, $\mu(U)\geq\alpha c$, and, if $\alpha<0$, $\nu(U)\leq\alpha c$.
It follows from the fact that, for $x_n\leq x<x_{n+1}$,
$$
\frac{\log\left(U(x)\right)}{\log(x)}
=\alpha\frac{x_n\,\log\left(2\right)}{\log(x)}
=\alpha c\frac{\log\left(x_{n+1}\right)}{\log(x)}
\left\{
\begin{array}{cll}
> & \alpha c\textrm{,} & \textrm{if $\alpha>0$} \\
< & \alpha c\textrm{,} & \textrm{if $\alpha<0$.}
\end{array}
\right.
$$
Next, if $\alpha<0$ we have that $U$ is a tail of distribution.
Let us check that the rv having a tail of distribution $\barF=U$ has a finite $s$th moment whenever $0\leq s<-\alpha c$.

Let $s>0$ and denote $x_0=0$.  We have
$$
\int_0^\infty x^sdF(x) = \sum_{n=1}^\infty x_n^s\left(U(x_{n}^-)-U(x_n^+)\right)
=\sum_{n=1}^\infty x_n^s\left(2^{\alpha x_{n-1}}-2^{\alpha x_n}\right)
\leq \sum_{n=1}^\infty 2^{(s/c-\alpha) x_{n-1}}<\infty
$$
because $s<-\alpha c$.

If $s=0$, let $\epsilon=-\alpha c/2$ ($>0$) and the statement follows from $\displaystyle \int_0^\infty dF(x)=\int_0^1 dF(x)+\int_1^\infty dF(x)\leq\int_0^1 dF(x)+\int_1^\infty x^\alpha dF(x)<\infty$.

Note that if $s\geq-\alpha c$, $\displaystyle \int_0^\infty x^sdF(x)=\infty$.
\end{proof}

\section{Proofs of results given in Section 2}

\subsection{Section 2.1}  \label{ProofsofSection2.1}

Let us introduce the following functions that will be used in the proofs.

We define, for some $b>0$ and $r\in\Rset$,
\begin{equation}\label{eq:20140405:002}
V_r(x)=\left\{
\begin{array}{ll}
\int_b^xy^rU(y)dy\text{,} & x\geq b \\
1\text{,} & 0< x<b\textrm{}
\end{array}
\right.
\quad\textrm{;}\quad
W_r(x)=\left\{
\begin{array}{ll}
\int_x^{\infty}y^rU(y)dy\text{,} & x\geq b \\
1\text{,} & 0< x<b\textrm{}
\end{array}
\right.
\end{equation}
For the main result, we will need the following lemma which is of interest on its own.
\begin{lem}\label{lem:kar:001}
Let $U\in\mathcal{M}$ with finite $\mathcal{M}$-index $\kappa_U$ and let $b>0$.
\begin{enumerate}
\item[(i)]
Consider $V_r$ defined in (\ref{eq:20140405:002}) with $r+1>\kappa_U$.
Then $V_r\in\mathcal{M}$ and its $\mathcal{M}$-index $\kappa_{V_r}$ satisfies $\kappa_{V_r}=\kappa_U-(r+1)$.
\item[(ii)]
Consider $W_r$ defined in (\ref{eq:20140405:002}) with $r+1<\kappa_U$.
Then $W_r\in\mathcal{M}$ and its $\mathcal{M}$-index $\kappa_{W_r}$ satisfies $\kappa_{W_r}=\kappa_U-(r+1)$.
\end{enumerate}
\end{lem}

\begin{proof}[Proof of Theorem \ref{prop:kar:001}] \hfill

\begin{itemize}
\item
\emph{Proof of the necessary condition of (K1$^*$)}

As an immediate consequence of Lemma \ref{lem:kar:001}, (i), we have, assuming that $\rho+r>0$:
\begin{eqnarray*}
\lefteqn{U\in\M\textrm{ with $\M$-index $\kappa_U=-\rho$ such that $(r-1)+1=r>-\rho=\kappa_U$}} \\
 & & \Longrightarrow\quad V_{r-1}(x)=\int_b^xt^{r-1}U(t)dt\in\M\textrm{ with $\M$-index $\kappa_{V_{r-1}}=\kappa_U-r=-\rho-r$}
\end{eqnarray*}
Hence, by applying Theorems \ref{teo:main:001} and \ref{teo:main:002} to $V_{r-1}$, the result follows:
$$
\limx\frac{\log\left(\int_b^xt^{r-1}U(t)dt\right)}{\log(x)}=\limx\frac{\log\left(V_{r-1}(x)\right)}{\log(x)}=-\kappa_{V_{r-1}}=\rho+r>0\textrm{.}
$$
\item
\emph{Proof of the sufficient of (K1$^*$)}

Using ($C1r$) and 
$
\displaystyle
\lim_{x\rightarrow\infty}\frac{\log\left(\int_b^xt^{r-1}U(t)dt\right)}{\log(x)}=\rho+r
$
gives
\begin{eqnarray*}
\lefteqn{\limx-\frac{\log\left(U(x)\right)}{\log(x)} =
\limx-\frac{\log\left(\frac{x^rU(x)}{\int_b^xt^{r-1}U(t)dt}\right)+\log\left(x^{-r}\int_b^xt^{r-1}U(t)dt\right)}{\log(x)}} \\
 & & \qquad\qquad=r+\limx-\frac{\log\left(\int_b^xt^{r-1}U(t)dt\right)}{\log(x)} = r-(\rho+r)=-\rho
\end{eqnarray*}
and the statement follows.
\item
\emph{Proof of the necessary condition of (K2$^*$)}

As an immediate consequence of Lemma \ref{lem:kar:001}, (ii), we have, assuming that $\rho+r<0$:
\begin{eqnarray*}
\lefteqn{U\in\M\textrm{ with $\M$-index $\kappa_U=-\rho$ such that $(r-1)+1=r<-\rho=\kappa_U$}} \\
 & & \Longrightarrow\quad W_{r-1}(x)=\int_x^\infty t^{r-1}U(t)dt\in\M\textrm{ with $\M$-index $\kappa_{W_{r-1}}=\kappa_U-r=-\rho-r$}
\end{eqnarray*}
Hence, by applying Theorems \ref{teo:main:001} and \ref{teo:main:002} to $W_{r-1}$, the result follows:
$$
\limx\frac{\log\left(\int_x^\infty t^{r-1}U(t)dt\right)}{\log(x)}=\limx\frac{\log\left(W_{r-1}(x)\right)}{\log(x)}=-\kappa_{W_{r-1}}=\rho+r<0\textrm{.}
$$
\item
\emph{Proof of the sufficient of (K2$^*$)}

Using ($C2r$) and 
$
\displaystyle
\lim_{x\rightarrow\infty}\frac{\log\left(\int_x^{\infty}t^{r-1}U(t)dt\right)}{\log(x)}=\rho+r
$ gives
\begin{eqnarray*}
\lefteqn{\limx-\frac{\log\left(U(x)\right)}{\log(x)} =
\limx-\frac{\log\left(\frac{x^rU(x)}{\int_x^{\infty}t^{r-1}U(t)dt}\right)+\log\left(x^{-r}\int_x^{\infty}t^{r-1}U(t)dt\right)}{\log(x)}} \\
 & & \qquad\qquad=r+\limx-\frac{\log\left(\int_x^{\infty}t^{r-1}U(t)dt\right)}{\log(x)} = r-(\rho+r)=-\rho
\end{eqnarray*}
and the statement follows.
\item
\emph{Proof of the necessary condition of (K3$^* $); case $\displaystyle \int_b^{\infty}t^{r-1}U(t)dt=\infty$ with $b>1$.}


On one hand, assumed $\rho+r=0$, $U\in\M$ with $\M$-index $\kappa_U=-\rho$ implies, for any $\epsilon>0$,
\begin{equation}\label{eq:20140402:101}
\lim_{x\rightarrow\infty}\frac{U(x)}{x^{\rho+\epsilon}}=0\textrm{\ \ \ and \ \ }\lim_{x\rightarrow\infty}\frac{U(x)}{x^{\rho-\epsilon}}=\infty\textrm{}
\end{equation}
On the other hand, $\displaystyle \int_b^{\infty}t^{r-1}U(t)dt=\infty$ implies
%
$
\displaystyle \lim_{x\rightarrow\infty}\int_b^xt^{r-1}U(t)dt=\infty\textrm{.}
$
Hence we can apply \thelr\, to the first limit of (\ref{eq:20140402:101}) to get, for any $\epsilon>0$,
\begin{equation}\label{eq:20140402:002}
\lim_{x\rightarrow\infty}\frac{\int_b^xt^{r-1}U(t)dt}{x^{\epsilon}}=\lim_{x\rightarrow\infty}\frac{x^{r-1}U(x)}{\epsilon x^{-1+\epsilon}}=
\lim_{x\rightarrow\infty}\frac{U(x)}{\epsilon x^{-r-1+\epsilon}}=\lim_{x\rightarrow\infty}\frac{U(x)}{\epsilon x^{\rho+\epsilon}}=0\textrm{}
\end{equation}
Moreover, we have, for any $\epsilon>0$,
\begin{equation}\label{eq:20140402:003}
\lim_{x\rightarrow\infty}\frac{\int_b^xt^{r-1}U(t)dt}{x^{-\epsilon}}=
\left(\lim_{x\rightarrow\infty}\int_b^xt^{r-1}U(t)dt\right)\,
\left(\lim_{x\rightarrow\infty}x^{\epsilon}\right)
=\infty\times\infty
=\infty\textrm{}
\end{equation}
Defining $V_{r-1}$ as in (\ref{eq:20140405:002}) 
we deduce from (\ref{eq:20140402:002}) and (\ref{eq:20140402:003}) that $V_{r-1}\in\M$ with $\M$-index $0=\rho+r$.
So, 
taking $x\geq b$, the required result follows:
$$
\lim_{x\rightarrow\infty}\frac{\log\left(\int_b^{x}t^{r-1}U(t)dt\right)}{\log(x)}=\lim_{x\rightarrow\infty}\frac{\log\left(V_{r-1}(x)\right)}{\log(x)}=\rho+r=0
$$
\item \emph{Proof of the necessary condition of (K3$^* $); case $\displaystyle \int_b^{\infty}t^{r-1}U(t)dt<\infty$ with $b>1$.}

Suppose $U\in\M\textrm{ with $\M$-index $\kappa_U=-\rho$}$. By a straightforward computation we have
$$
\lim_{x\rightarrow\infty}\frac{\log\left(\int_b^xt^{r-1}U(t)dt\right)}{\log(x)}=
\frac{\log\left(\int_b^{\infty}t^{r-1}U(t)dt\right)}{\lim_{x\rightarrow\infty}\log(x)}=0=\rho+r\textrm{}
$$
\item \emph{Proof of the sufficient condition of (K3$^* $)}

A similar proof used to prove the sufficient condition of (K1$^{* }$).
\end{itemize}
\vspace{-4ex}
\end{proof}

\begin{proof}[Proof of Lemma \ref{lem:kar:001}]~

\begin{itemize}
\item \emph{Proof of (i)}

Let us prove that $V_r$ defined in (\ref{eq:20140405:002}) belongs to $\M$ with $\M$-index $\kappa_{V_r}=\kappa_U-(r+1)$.

Choose 
$\rho=-\kappa_U+r+1>0$ and $0<\epsilon<\rho$.
Note that $x^{\rho\pm\epsilon}\to\infty$ as $\rho\pm\epsilon>0$.

Combining, for $x>1$, under the assumption $r+1>\kappa_U$, and for $U\in\M$,
$$
\lim_{x\rightarrow\infty}V_r(x)=\int_b^1y^rU(y)dy+\int_1^{\infty}y^rU(y)dy=\infty\textrm{,}
$$
and,
$$
\displaystyle
\limx\frac{\left(V_r(x)\right)'}{\left(x^{\rho+\delta}\right)'}
=\limx\frac{U(x)}{(\rho+\delta)x^{-\kappa_U+\delta}}
=\left\{
\begin{array}{ll}
0 & \textrm{if $\delta=\epsilon$} \\
\infty & \textrm{if $\delta=-\epsilon$}
\end{array}
\right.
$$
provides, applying \thelr,
$$
\displaystyle
\limx\frac{V_r(x)}{x^{\rho+\delta}}
=\limx\frac{\left(V_r(x)\right)'}{\left(x^{\rho+\delta}\right)'}
=\left\{
\begin{array}{ll}
0 & \textrm{if $\delta=\epsilon$} \\
\infty & \textrm{if $\delta=-\epsilon$,}
\end{array}
\right.
$$
which implies that $V_r\in\mathcal{M}$ with $\mathcal{M}$-index $\kappa_{V_r}=-\rho=\kappa_U-(r+1)$, as required.
\item \emph{Proof of (ii)}

First let us check that $W_r$ is well-defined.
Let $\delta=(\kappa_U-r-1)/2$ ($>0$ by assumption).
We have, for $U\in\mathcal{M}$,
$
\displaystyle
\lim_{x\rightarrow\infty}\frac{U(x)}{x^{-\kappa_U+\delta}}=0\textrm{,}
$
which implies that for $c>0$ there exists $x_0\geq1$ such that for all $x\geq x_0$, 
$
\displaystyle
\frac{U(x)}{x^{-\kappa_U+\delta}}\leq c\textrm{}
$.

Hence, one has, $\forall$ $x\geq x_0$,
$$
\int_x^{\infty}y^rU(y)dy
\leq c\int_x^{\infty}y^{-\kappa_U+\delta+r}dy
=c\int_x^{\infty}y^{\frac{-\kappa_U+r+1}{2}-1}dy
<\infty
$$
because of $-\kappa_U+r+1<0$.
Then, we can conclude, $U$ being bounded on finite intervals, that $W_r$ is well-defined.

Now choose $\rho=-\kappa_U+r+1<0$ and $0<\epsilon<-\rho$.
We have $x^{\rho\pm\epsilon}\to0$ as $\rho\pm\epsilon<0$.
We will proceed as in (i).

For $x>1$, under the assumption $r+1<\kappa_U$, for $U\in\M$, we have
$
\displaystyle
\lim_{x\rightarrow\infty}W_r(x)=\int_x^\infty y^rU(y)dy=0\textrm{,}
$
and 
$
\displaystyle
\limx\frac{\left(W_r(x)\right)'}{\left(x^{\rho+\delta}\right)'}
=\limx-\frac{U(x)}{(\rho+\delta)x^{-\kappa+\delta}}
=\left\{
\begin{array}{ll}
0 & \textrm{if $\delta=\epsilon$} \\
\infty & \textrm{if $\delta=-\epsilon$}
\end{array}
\right.
$.

Hence applying \thelr\  gives
$$
\displaystyle
\limx\frac{W_r(x)}{x^{\rho+\delta}}
=\limx\frac{\left(W_r(x)\right)'}{\left(x^{\rho+\delta}\right)'}
=\left\{
\begin{array}{ll}
0 & \textrm{if $\delta=\epsilon$} \\
\infty & \textrm{if $\delta=-\epsilon$,}
\end{array}
\right.
$$
which implies that $W_r\in\mathcal{M}$ with $\mathcal{M}$-index $\kappa_{W_r}=-\rho=\kappa_U-(r+1)$.
\end{itemize}
\vspace{-4ex}
\end{proof}

\subsection{Section 2.2} \label{ProofsofSection2.2}

\begin{proof}[Proof of Theorem \ref{teo:KaramatasTauberianTheoremExtension}]~

\vspace{-2ex}

\begin{itemize}
\item\emph{Proof of (i)}

Changing the order of integration in \eqref{eq:20140328:001}, using the continuity of $U$ and the assumption $U(0^+)=0$, give, for $s>0$,
$$
\widehat{U}(s)=s\int_{(0;\infty)}e^{-xs}U(x)dx \, ,
$$
%
or, with the change of variable $y=x/s$,
$$
\widehat{U}\left(\frac{1}{s}\right)=\int_{(0;\infty)}e^{-y}U(sy)dy\textrm{.}
$$
Let $U\in\M$ with $\M$-index $(-\alpha)<0$.
Let $0<\epsilon< \alpha$.

We have, via Theorems \ref{teo:main:001} and \ref{teo:main:002}, that there exists $x_0>1$ such that, for $x\geq x_0$,
$$
x^{\alpha-\epsilon}\leq U(x)\leq x^{\alpha+\epsilon}\textrm{.}
$$
Hence, for $s>1$, we can write
$$
\int_{x_0/s}^{\infty}e^{-x}(xs)^{\alpha-\epsilon}dx\leq\int_{x_0/s}^{\infty}e^{-x}U(xs)dx\leq\int_{x_0/s}^{\infty}e^{-x}(xs)^{\alpha+\epsilon}dx
$$
so
$
\displaystyle
\frac{\int_0^{x_0/s}e^{-x}U(xs)dx+\int_{x_0/s}^{\infty}e^{-x}x^{\alpha-\epsilon}dx}{s^{-\alpha+\epsilon}}\leq\widehat{U}\left(\frac{1}{s}\right)
\leq \frac{\int_0^{x_0/s}e^{-x}U(xs)dx+\int_{x_0/s}^{\infty}e^{-x}x^{\alpha+\epsilon}dx}{s^{-\alpha-\epsilon}}\textrm{}
$,
from which we deduce that
$$
-\alpha-\epsilon\leq\lim_{s\rightarrow\infty}-\frac{\log\left(\widehat{U}(1/s)\right)}{\log(s)}\leq-\alpha+\epsilon\textrm{.}
$$
Then we obtain, $\epsilon$ being arbitrary,
$
\displaystyle
\lim_{s\rightarrow\infty}-\frac{\log\left(\widehat{U}(1/s)\right)}{\log(s)}=-\alpha\textrm{.}
$

The conclusion follows, applying Theorem \ref{teo:main:001}, to get $\widehat{U}\circ g\in\M$ with $g(s)=1/s$, ($s>0$), and, Theorem \ref{teo:main:002}, 
for the $\M$-index.
\item\emph{Proof of (ii)}

Let $0<\epsilon< \alpha$. 

Since we assumed $U(0^+)=0$, we have, for $s>1$,
\begin{equation}\label{eq:20141107:010}
e^{-1}U(s)
 \leq \int_{(0;s)}e^{-\frac{x}{s}}dU(x)
 \leq \int_{(0;\infty)}e^{-\frac{x}{s}}dU(x) = \widehat{U}\left(\frac{1}{s}\right).
\end{equation}
Changing the order of integration in the last integral (on the right hand side of the previous equation), and using the continuity of $U$ and the fact that $U(0^+)=0$, gives, for $s>0$,
\begin{equation}\label{eq:20141107:011}
\widehat{U}\left(\frac{1}{s}\right)=\int_{(0;\infty)}e^{-x}U(sx)dx\textrm{.}
\end{equation}
Set $\displaystyle I_{\eta}=\int_{(0;\infty)}e^{-x}x^\eta dx$, for $\eta\in [0,\alpha)$ (such that $x^{-\eta}U(x)$ concave, by assumption).
Introducing the function $\displaystyle V(x):=I_{\eta}\,(sx)^{-\eta}\,U(sx)$, which is concave, and the rv $Z$ having the probability density function defined on $\Rset^+$ by $\displaystyle e^{-x}x^\eta\big/I_\eta$, we can write
$$
\int_{(0;\infty)}e^{-x}U(sx)dx = s^\eta \int_{(0;\infty)}  V(x)\, \frac{e^{-x}x^\eta}{I_\eta} dx = s^\eta \,E[V(Z)] \le s^\eta \,V\left(E[Z]\right) 
$$
applying Jensen's inequality. 
Hence we obtain, using that $\displaystyle E[Z]=I_{\eta+1}\big/I_{\eta}$ and the definition of $V$, 
$$
\int_{(0;\infty)}e^{-x}U(sx)dx \le \frac{I_{\eta}^{\,\eta+1}}{I_{\eta +1}^{\,\eta}}\,U\left(s\,I_{\eta+1}\big/I_{\eta}\right),
$$
from which we deduce, using \eqref{eq:20141107:011}, that
\begin{equation*} 
\frac1{s^{\alpha-\epsilon}}\widehat{U}\left(\frac{1}{s}\right) \le \frac{I_{\eta}^{\,\eta+1-\alpha+\epsilon}}{I_{\eta +1}^{\,\eta-\alpha+\epsilon}}\times \frac{U\left(s\,I_{\eta+1}\big/I_{\eta}\right)}{\left(s\,I_{\eta+1}\big/I_{\eta} \right)^{\alpha-\epsilon}} \,.
\end{equation*}
Therefore, since $\widehat{U}\circ g\in\M$ with $g(s)=1/s$ and $\M$-index ($-\alpha$), we obtain
$$ 
\frac{I_{\eta}^{\,\eta+1-\alpha+\epsilon}}{I_{\eta +1}^{\,\eta-\alpha+\epsilon}}\times \frac{U\left(s\,I_{\eta+1}\big/I_{\eta}\right)}{\left(s\,I_{\eta+1}\big/I_{\eta} \right)^{\alpha-\epsilon}}
\, \underset{s\to\infty}{\longrightarrow}\, \infty \, .
$$
But $\widehat{U}\circ g\in\M$ with $\M$-index ($-\alpha$) also implies in  \eqref{eq:20141107:010} that $\displaystyle \frac{e^{-1}U(s)}{s^{\alpha + \epsilon}}\, \underset{s\to\infty}{\longrightarrow} \, 0$.\\
From these last two limits, we obtain that $U\in\M$ with $\M$-index $(-\alpha)$.
\end{itemize}
\vspace{-4ex}
\end{proof}

\subsection{Section \ref{Section2.3}} \label{ProofsofSection2.3}

\vspace{-2ex}
\begin{proof}[Proof of Proposition \ref{prop:main:vonmises:001}]~

\vspace{-2ex}

\begin{itemize}
\item \emph{Proof of (i)}

Suppose that $F$ satisfies 
$
\displaystyle \lim_{x\rightarrow\infty}\frac{x\,F'(x)}{\overline{F}(x)}=\alpha
$.
Applying the L'H\^{o}pital's rule gives
%
$$
\lim_{x\rightarrow\infty}\frac{x\,F'(x)}{\overline{F}(x)}=
\lim_{x\rightarrow\infty}-\frac{\left(\log\left(\overline{F}(x)\right)\right)'}{\left(\log(x)\right)'}=
\lim_{x\rightarrow\infty}-\frac{\log\left(\overline{F}(x)\right)}{\log(x)}=
\frac{1}{\alpha}\textrm{,}
$$
hence $\overline{F}\in\M$, via Theorem \ref{teo:main:001}, with $\M$-index $\kappa_{\barF}=1/\alpha$, via Theorem \ref{teo:main:002}.
\item
\emph{Proof of (ii)}

Suppose that $F$ satisfies $\displaystyle \limx\left(\frac{\overline{F}(x)}{F'(x)}\right)'$=0.
It implies that, for all $\epsilon>0$, there exists $x_0>0$ such that, for $x\geq x_0$,
$
\displaystyle
-\epsilon\leq\left(\frac{\overline{F}(x)}{F'(x)}\right)'\leq\epsilon\textrm{.}
$

Integrating this inequality on $[x_0,x]$ gives
$$
-\epsilon(x-x_0)\leq\left(\frac{\overline{F}(x)}{F'(x)}\right)-\left(\frac{\overline{F}(x_0)}{F'(x_0)}\right)\leq\epsilon(x-x_0)\textrm{}
$$
from which we deduce
%
$
\displaystyle -\epsilon\leq\lim_{x\rightarrow\infty}\frac{\overline{F}(x)}{xF'(x)}\leq\epsilon\textrm{,}
$
hence 
%
$
\displaystyle
\lim_{x\rightarrow\infty}\frac{\overline{F}(x)}{xF'(x)}=0\textrm{.}
$
$$
\lim_{x\rightarrow\infty}\frac{x\,F'(x)}{\overline{F}(x)}=
\lim_{x\rightarrow\infty}-\frac{\left(\log\left(\overline{F}(x)\right)\right)'}{\left(\log(x)\right)'}=
\lim_{x\rightarrow\infty}-\frac{\log\left(\overline{F}(x)\right)}{\log(x)}=
\frac{1}{0}=\infty\textrm{,}
$$
since $F'(x)>0$ as $x\to\infty$.

We conclude that $\overline{F}\in\M_\infty$, via Theorem \ref{teo:main:001extension}.
\end{itemize}
\vspace{-2ex}
\end{proof}

\vspace{-4ex}

\begin{proof}[Proof of Theorem \ref{teo:20140411:001}]~


\begin{itemize}
\item
Let $F\in DA(\Phi_\alpha)$, $\alpha>0$.
Then Theorem \ref{teo:mises:dehaanferreira0} and Proposition \ref{prop:20140329:strictsubset} imply that $\overline{F}\in RV_{-\alpha}\subseteq\M$ with $\M$-index $\kappa_{\barF}=-\alpha$.
\item
Assume $F\in DA(\Lambda_{\infty})$.
Applying Corollary \ref{cor:mises:dehaan0} gives
$
\displaystyle \lim_{x\rightarrow\infty}-\frac{\log\left(\overline{F}(x)\right)}{\log(x)}=\infty
$.
Theorem \ref{teo:main:001extension} allows to conclude. 
\end{itemize}
\vspace{-2ex}
\end{proof}

\vspace{-4ex}

\begin{proof}[Proof of Example \ref{exm:20140925:001}]~

Let us check that $F\not \in DA(\Lambda_{\infty})$. 

We prove it by contradiction.

Suppose that $F$ defined in \eqref{eq:20140526:002} belongs to $DA(\Lambda_{\infty})$.
By applying Theorem \ref{teo:mises:gnedenko:20140412:001}, we conclude that there exists a function $A$ such that $A(x)\rightarrow0$ as $x\rightarrow\infty$ and \eqref{eq:mises:20140412:010bis} holds.

Introducing the definition \eqref{eq:20140526:002} into \eqref{eq:mises:20140412:010bis}, we can write, for all $x\in\Rset$,
\begin{eqnarray}
\lefteqn{\lim_{z\rightarrow\infty}\Big(\lfloor z\,(1+A(z)\,x)\rfloor\,\log\big(z\,(1+A(z)\,x)\big)-\lfloor z\rfloor\,\log(z)\Big)} \nonumber \\
 & & =\lim_{z\rightarrow\infty}\Big(\big(\lfloor z\,(1+A(z)\,x)\rfloor-\lfloor z\rfloor\big)\log\left(z\right)+\lfloor z\,(1+A(z)\,x)\rfloor\,\log\big(1+A(z)\,x\big)\Big)=x\textrm{}
\label{eq:20140413:001}
\end{eqnarray}
Let us see that the assumption of the existence of such function $A$ leads to a contradiction when considering some values $x$.
\begin{itemize}
\item Suppose $\displaystyle \lim_{z\to\infty} z\,A(z)=c>0$.

Take $x>0$ such that $cx/2\geq1$ and $z$ large enough such that $z\,A(z)\geq c/2$.

On one hand, we have
$
\lfloor z\,(1+A(z)\,x)\rfloor-\lfloor z\rfloor>0
$
since
$
z\,(1+A(z)\,x)
\geq
z+cx/2
\geq
z+1
$.
This implies that
$$
\lim_{z\rightarrow\infty}\big(\lfloor z\,(1+A(z)\,x)\rfloor-\lfloor z\rfloor\big)\log\left(z\right)=\infty\textrm{.}
$$
On the other hand, we have, taking $z$ large enough to have $\log\big(1+A(z)\,x\big)\approx A(z)\,x$ and $z\,A(z)\leq 2c$,
\begin{eqnarray*}
\lefteqn{\lfloor z\,(1+A(z)\,x)\rfloor\,\log\big(1+A(z)\,x\big)} \\
 & & \leq
z\,(1+A(z)\,x)\,\log\big(1+A(z)\,x\big)
\approx
z\,(1+A(z)\,x)\,A(z)\,x
\leq
2\,c\,(1+A(z)\,x)\,x<\infty
\end{eqnarray*}
Combining these results and taking $z\to\infty$ contradict \eqref{eq:20140413:001}.

\item Suppose $\displaystyle \lim_{z\to\infty} z\,A(z)=0$.

Let $x>0$.

On one hand, we have that
$$
\lim_{z\rightarrow\infty}\big(\lfloor z\,(1+A(z)\,x)\rfloor-\lfloor z\rfloor\big)\log\left(z\right)
$$
could be 0 or $\infty$ depending on the behavior of $z\,A(z)$ as $z\to\infty$.

On the other hand, we have, taking $z$ large enough such that $\log\big(1+A(z)\,x\big)\approx A(z)\,x$,
\begin{eqnarray*}
\lefteqn{\lfloor z\,(1+A(z)\,x)\rfloor\,\log\big(1+A(z)\,x\big)} \\
 & & \leq
z\,(1+A(z)\,x)\,\log\big(1+A(z)\,x\big)
\approx
z\,(1+A(z)\,x)\,A(z)\,x
\to0
\quad\textrm{as}
\quad
z\to\infty\textrm{.}
\end{eqnarray*}
Combining these results contradicts \eqref{eq:20140413:001}.
\end{itemize}
\end{proof}


\begin{thebibliography}{99}
\bibitem{BalkemadeHaan1974} \textsc{A. Balkema, L. de~Haan}, Residual Life Time at Great Age. \emph{Ann. Probab.} {\bf 2}, (1974) 792-804.
\bibitem{BasrakDavisMikosch2002} \textsc{B. Basrak, R. Davis, T. Mikosch}, A Characterization of Multivariate Regular Variation. \emph{Ann. Appl. Probab.} {\bf 12}, (2002) 908-920.
\bibitem{BinghamGoldieOmey2006} \textsc{N. Bingham, C. Goldie, E. Omey}, Regularly varying probability densities. \emph{Publications de l'Institut Math\'{e}matique} {\bf 80}, (2006) 47-57.
\bibitem{BinghamGoldieTeugels} \textsc{N. Bingham, C. Goldie, J. Teugels}, Regular Variation. \emph{Cambridge University Press} (1989).
\bibitem{Daley2001} \textsc{D. Daley}, The Moment Index of Minima. \emph{J. Appl. Probab.} {\bf 38}, (2001) 33-36.
\bibitem{DaleyGoldie2006} \textsc{D. Daley, C. Goldie}, The moment index of minima (II). \emph{Stat. \& Probab. Letters} {\bf 76}, (2006) 831-837.
\bibitem{deHaan} \textsc{L. de~Haan}, On regular variation and its applications to the weak convergence of sample extremes. \emph{Mathematical Centre Tracts, {\bf 32}} (1970).
\bibitem{deHaan1971} \textsc{L. de~Haan}, A Form of Regular Variation and Its Application to the Domain of Attraction of the Double Exponential Distribution. \emph{Z. Wahrsch. V. Geb.} {\bf 17}, (1971) 241-258.
\bibitem{deHaanFerreira} \textsc{L. de~Haan, A. Ferreira}, Extreme Value Theory. An Introduction. \emph{Springer}, (2006).
\bibitem{deHaanResnick1981} \textsc{L. de~Haan, S. Resnick}, On the observation closest to the origin. \emph{Stoch. Proc. Applic.} {\bf 11}, (1981) 301-308.
\bibitem{embrechts1997} \textsc{P. Embrechts, C. Kl\"{u}ppelberg, T. Mikosch}, Modelling Extremal Events for Insurance and Finance. \emph{Springer Verlag} (1997).
\bibitem{EmbrechtsOmey1984} \textsc{P. Embrechts, E. Omey}, A property of longtailed distributions. \emph{J. Appl. Probab.} {\bf 21}, (1984) 80-87.
\bibitem{feller21966} \textsc{W. Feller}, An introduction to probability theory and its applications. Vol II. \emph{J. Wiley \& Sons} (1966).
\bibitem{FisherTippett1928} \textsc{R. Fisher, L. Tippett}, Limiting forms of the frequency distribution of the largest or smallest number of a sample. \emph{Proc. Cambridge Phil. Soc.} {\bf 24}, (1928) 180-190.
\bibitem{Gnedenko1943} \textsc{B. Gnedenko}, Sur La Distribution Limite Du Terme Maximum D'Une S\'{e}rie Al\'{e}atoire. \emph{Ann. Math.} {\bf 44}, (1943) 423--453.
\bibitem{Goldie1978} \textsc{C. Goldie}, Subexponential distributions and dominated-variation tails. \emph{J. Appl. Probab.} {\bf 15}, (1978) 440-442.
\bibitem{HadarRussell1971} \textsc{J. Hadar, W. Russell}, Stochastic Dominance and Diversification. \emph{J. Econ. Theory} {\bf 3}, (1971) 288-305.
\bibitem{Hill1975} \textsc{B. Hill}, A Simple General Approach to Inference About the Tail of a Distribution. \emph{Ann. Stat.} {\bf 3}, (1975) 1163-1174.
\bibitem{Karamata1930} \textsc{J. Karamata}, Sur un mode de croissance r\'{e}guli\`{e}re des fonctions. \emph{Mathematica (Cluj)} {\bf 4}, (1930) 38-53.
\bibitem{Karamata1931} \textsc{J. Karamata}, Neuer Beweis und Verallgemeinerung der Tauberschen S\"{a}tze, welche die Laplacesche und Stieltjessche Transformation betreffen. \emph{J. R. A. Math.} {\bf 1931}, (1931) 27-39.
\bibitem{Karamata1931a001} \textsc{J. Karamata}, Sur le rapport entre les convergences d'une suite de fonctions et de leurs moments avec application \`{a} l'inversion des proc\'{e}d\'{e}s de sommabilit\'{e}. \emph{Studia Math.} {\bf 3}, (1931) 68-76.
\bibitem{Karamata1933} \textsc{J. Karamata}, Sur un mode de croissance r\'{e}guli\`{e}re. Th\'{e}or\`{e}mes fondamentaux. \emph{Bulletin SMF} {\bf 61}, (1933) 55-62.
\bibitem{Lindskog} \textsc{F. Lindskog}, Multivariate extremes and regular variation for stochastic processes. \emph{Ph.D. thesis, ETH Z\"urich, available online} (2004).
\bibitem{Mikosch} \textsc{T. Mikosch}, Non-Life Insurance Mathematics. An Introduction with Stochastic Processes. \emph{Springer} (2006).
\bibitem{Pickands1975} \textsc{J. Pickands}, Statistical Inference Using Extreme Order Statistics. \emph{Ann. Stat.} {\bf 3}, (1975) 119-131.
\bibitem{resnick3} \textsc{S. Resnick}, Extreme Values, Regular Variation, and Point Processes. \emph{Springer-Verlag} (1987).
\bibitem{Resnick2004} \textsc{S. Resnick}, On the Foundations of Multivariate Heavy-Tail Analysis. \emph{J.  Appl. Probab.} {\bf 41}, (2004) 191-212.
\bibitem{shaked} \textsc{M. Shaked, G. Shanthikumar}, Stochastic Orders. \emph{Springer} (2007).
\bibitem{Stankovic1990} \textsc{B. Stankovi\'{c}}, Regularly Varying Distributions. \emph{Public. Inst. Math.} {\bf 48}, (1990) 119-128.
\bibitem{vonMises1936} \textsc{R. von~Mises}, La distribution de la plus grande de $n$ valeurs. \emph{Revue Math. Union Interbalkanique} {\bf 1}, (1936) 141-160.
\end{thebibliography}
\end{document}